\documentclass[11pt]{amsart}

\usepackage{amsmath,amssymb,amsfonts,amsthm,amscd,indentfirst}

\usepackage{hyperref}

\usepackage{graphicx}
\usepackage{xcolor}

\usepackage{accents}

\makeatletter
\newcommand*{\rom}[1]{\expandafter\@slowromancap\romannumeral #1@}
\makeatother

\usepackage{amsmath,amssymb,amsfonts,amsthm,amscd, color}
\usepackage{indentfirst}

\theoremstyle{plain}
\newtheorem{theorem}{Theorem}[section]
\newtheorem{corollary}[theorem]{Corollary}
\newtheorem{lemma}[theorem]{Lemma}
\newtheorem{proposition}[theorem]{Proposition}

\theoremstyle{remark}
\newtheorem{remark}{Remark}[section]

\theoremstyle{definition}

\theoremstyle{definition}

\numberwithin{equation}{section}

\def\Ric{\operatorname{Ric}}

\def\R{\mathbb{R}}

\def\setminus{ \backslash}

\def\inf{\operatorname{inf}}

\newcommand{\be}{\begin{equation}}
\newcommand{\ee}{\end{equation}}
\newcommand{\bee}{\begin{equation*}}
\newcommand{\eee}{\end{equation*}}

\def\p{\partial}

\def\la{\langle}
\def\ra{\rangle}
\def\lf{\left}
\def\ri{\right}
\def\Pi{\displaystyle{\mathbb{II}}}

\def\vh{\vspace{.2cm}}

\def\m{\mathfrak{m}}

\def\div{\operatorname{div}}
\def\m{\mathfrak{m}}

\def\d{\frac {\operatorname d}{\operatorname {dt}}}

\def\mu{\nu} 
\def\S{\Sigma}

\def\Na{\nabla}
\def\eps{\varepsilon}

\def\d{\delta}

\title{Mass of asymptotically flat $3$-manifolds \\ with boundary}

\author{Sven Hirsch}
\address{Department of Mathematics, Duke University, Durham, NC 27708, USA}
\email{sven.hirsch@duke.edu}

\author{Pengzi Miao}
\address{Department of Mathematics, University of Miami, Coral Gables, FL 33146, USA}
\email{pengzim@math.miami.edu}
\thanks{P. Miao acknowledges the support of NSF Grant DMS-1906423.}

\author{Tin-Yau Tsang}
\address{Department of Mathematics, University of California, Irvine, CA 92697}
\email{tytsang@uci.edu}

\begin{document}

\begin{abstract}
We study the mass of asymptotically flat $3$-manifolds with boundary using the method of Bray-Kazaras-Khuri-Stern \cite{BKKS}.
More precisely,  we derive a mass formula on the union of an asymptotically flat manifold and fill-ins of its boundary, 
and give new sufficient conditions guaranteeing the positivity of the mass.
Motivation to such consideration comes from studying the quasi-local mass of the boundary surface.
If the boundary isometrically embeds in the Euclidean space,  
we apply the formula to obtain convergence of the Brown-York mass 
along large surfaces tending to $\infty$ which include  the scaling of any fixed coordinate-convex surface. 
\end{abstract}

\maketitle

\markboth{Sven Hirsch, Pengzi Miao and Tin-Yau Tsang}{Mass of asymptotically flat $3$-manifolds with boundary}

\section{Introduction}

Given a complete asymptotically flat $3$-manifold $(M^3, g)$, Bray, Kazaras, Khuri and Stern   \cite{BKKS} gave a new 
proof of the three dimensional Riemannian positive mass theorem \cite{SY1, W}. More precisely, their result shows
\be \label{eq-BKKS}
\m (g) \ge \frac{1}{16 \pi} \int_{M_{_{ext}} } \left( \frac{ | \nabla^2 u |^2}{ | \nabla u | } + R | \nabla u | \right) ,
\ee
where $\m (g)  $ is the total mass \cite{ADM} of $(M, g)$,  $M_{_{ext}}$ is the exterior region in $(M,g)$, $R$ denotes 
the scalar curvature of the metric, and $u$ is a harmonic function, satisfying a Neumann boundary condition at $\partial M_{_{ext}}$, 
and is asymptotic to one of the asymptotically flat coordinate functions at infinity.

The method in \cite{BKKS} was inspired by the level sets approach to study scalar curvature via harmonic maps initiated by Stern \cite{S}
(also see Bray-Stern \cite{BS}).
The method was  further explored by Kazaras, Khuri and the first named author \cite{HKK} to prove the three dimensional 
spacetime positive mass theorem \cite{SY2} and  produced a new lower bound on the total energy-momentum of initial data sets.

In this paper, we apply the method of Bray-Kazaras-Khuri-Stern \cite{BKKS} to analyze the mass of asymptotically flat $3$-manifolds 
$(M^3, g)$ with nonempty boundary $\Sigma$.
In the case of $M$ being diffeomorphic to $ \R^3 $ minus a ball, the mass of such manifolds is often used  in the study of 
the quasi-local mass of $\Sigma$ (see \cite{Bartnik-qlmass, ST, WY2} for instance).

The following formula identifies the boundary terms of  the integral in \cite{BKKS} that relates the mass  and harmonic functions. 
(See Proposition \ref{prop-mass-bdry} and Remark \ref{rem-boundary-angle} in Section 2.)

\begin{proposition} \label{prop-mass-bdry-intro}
Let $(M, g)$ be an asymptotically flat $3$-manifold with boundary $\Sigma$. 
Let $u $ be a  harmonic function on $M$ which is asymptotic to one of  the asymptotically flat coordinate functions at infinity.  
Let $ \Sigma_t = u^{-1}(t)$ be the level set of $u$ and  $\mu $ be the infinity pointing 
unit normal to $\Sigma$. Then
\be \label{intro-mass-bdry}
\begin{split}
& \ 8 \pi \m (g) +  2 \pi  \int_{- T }^{T}  ( \chi(\Sigma_t) - 1 )  \, dt \\
\ge & \ \int_{M}  \frac{1}{2} \left( \frac{ | \nabla^2 u |^2 }{  | \nabla u |  }  +  R | \nabla u |   \right) \\
& \  - \int_{\Sigma} H | \nabla u |  
  +  \int_{ \{  \nabla_{_{\Sigma} } u \ne 0 \} } \left \la \nabla_{_{\Sigma }} \beta , \frac{ \nabla_{_{\Sigma}} u }{ | \nabla_{_\Sigma } u | }  \right \ra | \nabla u | .
\end{split}
\ee
Here $T >0 $ is some large constant so that  the Euler characteristic  $ \chi(\Sigma_t) = 1 $ for  $ | t | \ge T$; 
$H$ is the mean curvature of $\Sigma$ with respect to $\mu$;
$\beta$ is the angle between $\nabla u $ and $\Sigma$, which equals zero if $ \frac{\p u}{\p \nu} = 0 $;  
and operators $ \Delta_{_\Sigma}$,  $ \nabla_{_{\Sigma}}$ and $ \nabla^2_{_\Sigma}$ denote 
the corresponding Laplacian, gradient, and Hessian on $\Sigma$. 
\end{proposition}

Using \eqref{intro-mass-bdry}, we derive a formula of the mass of $(M, g)$  via fill-ins of its boundary $\Sigma$.
We say a compact Riemannian $3$-manifold $(\Omega, g_{_\Omega})$ is a fill-in of $\Sigma$ 
if  $ \p \Omega$ is isometric to $ \Sigma$.

\begin{figure}[h]
\includegraphics[scale=0.6]{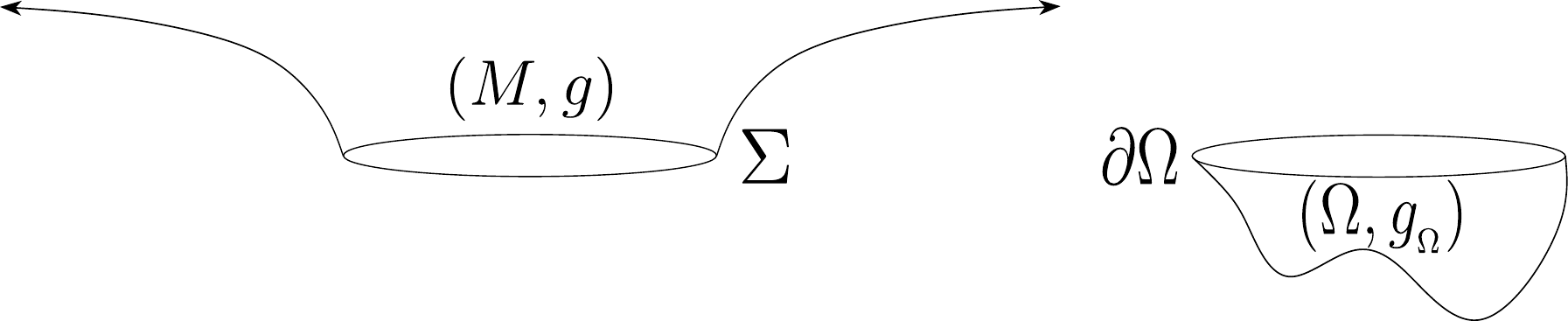}
  \caption{An asymptotically flat  manifold and  a fill-in of its boundary}
\end{figure}

Throughout the paper, $M$ and $ \Omega$ are assumed to be orientable and 
metrics $g$ and $g_{_\Omega}$ are assumed to be $C^2$ up to the boundary.

\begin{theorem}\label{main-1-intro}
Let $(M,g)$ be an asymptotically flat $3$-manifold with boundary $\Sigma$.
Let $ (\Omega, g_{_\Omega} )$ be a fill-in of $\Sigma$. 
Then there exist harmonic functions $ u$, $ u_{_\Omega}$ on $(M, g)$, $(\Omega, g_{_\Omega})$ respectively, so that
$u$ is asymptotic to one of the asymptotically flat coordinate functions near infinity on $(M, g)$, 
$u$ and $u_{_\Omega}$  are $C^2$ up to the boundary, 
and along the boundary, 
\be
u = u_{_\Omega} , \ \ \frac{\p u}{\p \mu} = \frac{ \p u_{_\Omega } }{\p \nu} .
\ee
Here $ \mu $ denotes the $\infty$-pointing unit normal to $ \Sigma$ in $(M, g)$, and by abuse of notation also the outward normal to 
$\p \Omega$ in $(\Omega, g_{_\Omega} )$. The equation is to be understood as two sides being equal under the isometry between $ \Sigma$ 
and $ \p \Omega$. 
If the topological manifold obtained by gluing $M$ and $ \Omega$ along their boundary
contains no non-separating $2$-spheres, then 
\begin{equation} \label{eq-main-intro}
\begin{split}
\m (g) \ge & \  \frac{1}{16 \pi} \left[ \int_{M   }\left(  \frac{|\nabla^2 u|}{|\nabla u|}+R|\nabla u|\right) 
+ \int_{ \Omega  }\left( \frac{|\nabla^2 u_{_\Omega} |}{|\nabla u_{_\Omega} |}+R|\nabla u_{_\Omega} |\right)   \right] \\
& \ + \frac{1}{8\pi} \int_\S (H_{_\Omega} - H ) |\nabla u| .
\end{split}
\end{equation}
Here $H$, $ H_{_\Omega}$ are the mean curvature of $\Sigma$, $ \p \Omega$ in $(M, g)$, $(\Omega, g_{_\Omega})$ respectively, 
and $ | \nabla u |$ denotes the common value of the length of the gradient of $ u$ and $ u_{_\Omega}$ at $\Sigma$ and $ \p \Omega$.
\end{theorem}

\begin{remark}
Given  $(M^3, g)$ with boundary $\Sigma$, a fill-in of $\Sigma$ modulo the smoothness at a  point
 can be produced using $(M, g)$ itself (see Section \ref{sec-conformal-fill-in} for details). 
 More generally, it is shown by Shi, Wang and Wei \cite{SWW} that fill-ins with 
 nonnegative scalar curvature always exist for
 $\Sigma$.  Related study of fill-ins with or without mean curvature constraint can be found 
 in \cite{J, JMT, MM, MMT, SWW, SWWZ}. 
\end{remark}

Formula \eqref{eq-main-intro} is not surprising in the sense that the mean curvature difference $(H_{_\Omega} - H )$
measures the distributional scalar curvature of the manifold $(N, \hat g) = (M, g) \cup (\Omega, g_{_\Omega} )$ 
across $ \Sigma = \p \Omega$ (see \cite{M1, ST}).  
A proof of \eqref{eq-main-intro} then can be given by 
approximating $\hat g$ by smooth metrics  produced in \cite{M1} 
and applying formula \eqref{eq-BKKS} to the approximation  (see the appendix). 
In Section \ref{sec-mass-fill-in} below, we choose a more direct approach toward Theorem \ref{main-1-intro} by 
working with harmonic functions defined on the singular manifold $(N, \hat g)$.
By analyzing their regularity at the corner surface $\Sigma $, in Proposition 3.1 we show these functions  
are necessarily smooth on $\Sigma $, and hence smooth up to $\Sigma$ from the two sides in $(N, \hat g)$. 
As a result, except the mean curvature term, boundary terms in \eqref{intro-mass-bdry} from $(M, g)$ are compensated by the 
corresponding boundary terms from $(\Omega, g_{_\Omega})$, which gives \eqref{eq-main-intro}.
We think this type of tangential regularity result can be useful in other PDE settings as well.

Formula \eqref{eq-main-intro} leads to a sufficient condition that guarantees positivity of $\m (g)$.

\begin{corollary} \label{cor-WY}
Let $(M, g)$ and $(\Omega, g_{_\Omega})$ be given as in Theorem \ref{main-1-intro}.
Suppose there exist vector fields $X$, $Y$ on $\Omega$, $M$, respectively, such that 
\begin{itemize}
\item[a)]  $ R \ge C_1|X|^2-2\div X \ \text{on} \ \Omega $,
\item[b)] $  R \ge C_2|Y|^2-2\div Y \ \text{on} \  M $ with $Y=O(|x|^{-1-2\tau})$ for some $\tau>\frac{1}{2}$ and
\item[c)] $ H-\la Y,\mu  \ra\le H_{_\Omega}-\la X,\nu \ra$.
\end{itemize} 
Here $C_1,C_2\geq\frac{2}{3}$ are some constants. Then $\m (g) \ge0$.
If $\m (g) = 0 $ and  $C_1, C_2>\dfrac{2}{3}$, then  $X=0$, $Y=0$,
$\Sigma$ and $ \p \Omega$ have the same second fundamental forms,   
and  the manifold $(M , g)\cup (\Omega, g_{_\Omega})$ obtained by gluing along $\Sigma = \p \Omega$ 
is isometric to $\R^3$.
\end{corollary}

\begin{remark}
Motivations for allowing such vector fields on $\Omega $ and $M$ come from the use of Jang's equation \cite{SY2}
in the context of quasi-local mass (see \cite{LY1, LY2, WY1, WY2} for instance). 
It is known the scalar curvature $R$ of the graph of solutions to Jang's equation satisfies  
$ R \ge 2 | X |^2 - 2 \div X$ for a vector filed $X$.
We are recently informed by Aghil Alaee that 
results similar to $\m (g) \ge 0$  in Corollary \ref{cor-WY} may also be derived via the method in \cite{Alaee-K-Y}.
\end{remark}

If we make use of a fill-in of $\Sigma $ produced by $(M, g)$ itself, we obtain a result that only depends on $(M, g)$.

\begin{corollary} \label{Green-intro}
Let $(M, g)$ be an asymptotically flat $3$-manifold with boundary $ \Sigma$. 
Let $ G$ be the harmonic function on $(M, g)$ satisfying 
$ G = 1 $ at $ \Sigma $ and $ G \to 0 $ at $ \infty $.
Let $ w $ be the harmonic function satisfying $ w = 0 $ at $ \Sigma$ and $ w $ is asymptotic to an asymptotically flat coordinate function 
at $\infty$. 
Then there exists a harmonic function $v$ on $(M,g)$ such that $ v \to 0 $ at $\infty$ and 
\be
 v \frac{\p G}{\p \mu} - 2 \frac{\p v}{\p \mu}  = \frac{ \p w}{\p \mu} \ \text{at} \ \Sigma .
\ee
Moreover, if $(M, g)$ is diffeomorphic to $ \R^3 $ minus a ball, 
\be \label{eq-m-Gw}
\m(g) \ge  \frac{1}{16 \pi} \int_{M   }\left[\frac{|\nabla^2u|}{|\nabla u|}+R|\nabla u|\right] -  \frac{1}{4\pi} \int_\S ( H + 2  \frac{\p G}{\p \mu} ) |\nabla u| .
\ee
Here $ u = v + w $.
\end{corollary}

\begin{remark}
If $ R \ge 0 $  and $ H \le - 2  \frac{\p G}{\p \mu} $,  \eqref{eq-m-Gw} implies $ \m (g) \ge 0 $, which is a special case of results in \cite{HM}.
\end{remark}

Given a surface $ \Sigma$ with mean curvature $H$ in a $3$-manifold $(M, g)$, if $\Sigma$ has positive Gauss
curvature, the Brown-York mass  (\cite{BY1, BY2}) of $\Sigma$ is given by
\be
\m_{_{BY}} (\Sigma) = \frac{1}{8\pi} \int_{\Sigma} (H_0 - H) . 
\ee
Here $ H_0$ denotes the mean curvature of the isometric embedding of $\Sigma$ in the Euclidean space $(\R^3, g_{_0})$.

In an asymptotically flat $3$-manifold $(M, g)$, let $ \Sigma$ denote a large $2$-sphere near infinity. It is 
a natural question to compare the Brown-York mass of  $\Sigma$ and the mass of $ (M, g)$.

\begin{figure}[h]
\includegraphics[scale=0.5]{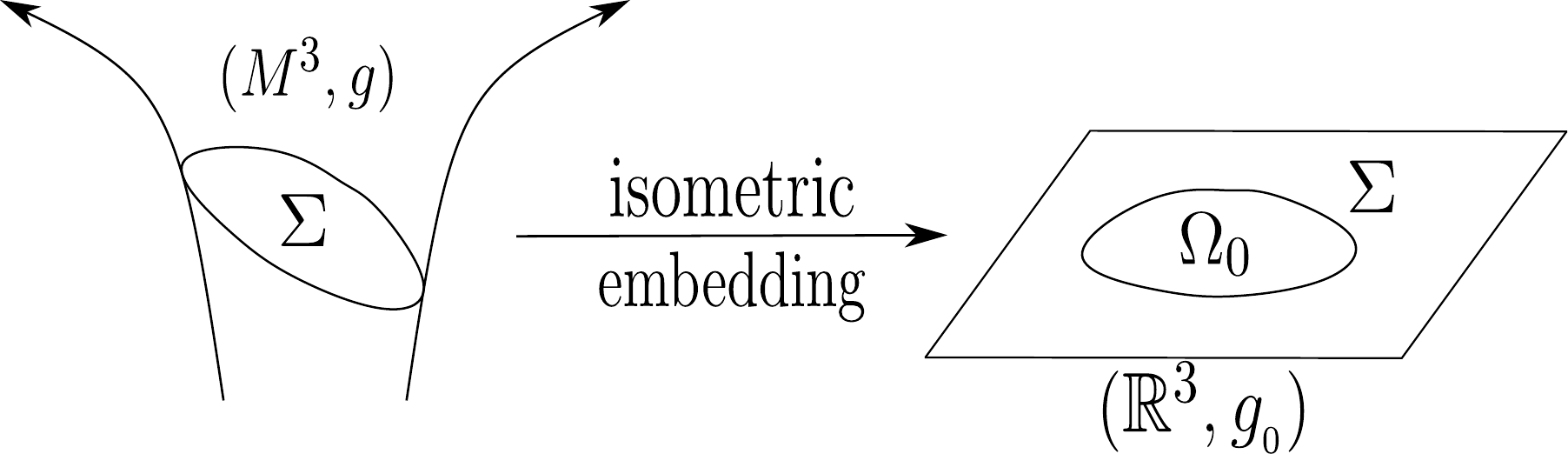}
  \caption{Brown-York mass of large surfaces in $(M, g)$}
\end{figure}

If $\{ \Sigma_r \}$ is a family of large coordinate spheres, it was demonstrated  in \cite{FST} that 
\be \label{eq-BY-mass-intro}
 \lim_{r \to \infty} \m_{_{BY}} (\Sigma_r) = \m (g) . 
\ee
The same convergence property was shown for nearly round surfaces in \cite{SWW-09} and for 
surfaces of revolution in asymptotically Schwarzschild manifold in \cite{FK-1, FK-2}.

Theorem \ref{main-1-intro} suggests a natural approach to analyze the convergence of $ \m_{_{BY} } (\Sigma_r)$ for 
a more general class of surfaces. Let $ \Omega_0$ be  the domain enclosed by $\Sigma$ in $\R^3$. 
Then $(\Omega_0, g_0)$ provides a natural fill-in of the asymptotically end $(E, g)$, 
where $ E$ is the exterior of $ \Sigma$ in $M$. Applying \eqref{eq-main-intro} to 
the manifold $(N, \hat g) = (E, g) \cup (\Omega_0, g_0)$, 
one expects to relate $\m(g)$ to $\m_{_{BY}} (\Sigma)$ by 
estimating the corresponding harmonic function. (In this case it can be shown that \eqref{eq-main-intro} is indeed an equation.)
By implementing this idea, we show that

\begin{theorem} \label{main-2-intro}
Let $(M, g)$ be an asymptotically flat $3$-manifold with a coordinate chart $\{ x_i \}$ in which 
the metric coefficients $g_{ij}$ satisfies the asymptotically flat condition. 
Let $ \{  \Sigma_r \}_{ r \ge r_0}  $ be a $1$-parameter family of surfaces approaching infinity,
where  $ r = \min_{\Sigma_r} \{ | x | \ | \ x \in \Sigma_r \} $.
Suppose the rescaled surfaces
$$
\widetilde \Sigma_r : = \left\{ \frac{1}{r} x \, | \, x \in \Sigma_r \right \} \subset \R^3
$$
satisfy the property
\be \label{eq-condition-kappa-intro}
0 < k_1 < \kappa_\alpha < k_2, \ \alpha = 1, 2, 
\ee
where $ \{ \kappa_\alpha \}$ are the principal curvatures of $ \widetilde \Sigma_r $ in $ \R^3$ with respect to the background Euclidean metric 
and $ k_1, k_2$ are two positive constants independent of $r$. 
Then 
\be \label{eq-BY-limit-1-intro}
 \lim_{r \to \infty} \m_{_{BY} }( \Sigma_r) = \m (g). 
 \ee
\end{theorem}

In particular, Theorem \ref{main-2-intro} applies to any family of surfaces that are
obtained by scaling a fixed convex surface in the background coordinate space.

We end this introduction with a connection of \eqref{eq-main-intro} to the positivity of Brown-York mass.
Suppose $ \Omega \subset \R^3$ is a bounded domain with smooth boundary $\Sigma$. Let $g$ be a metric 
on $ \Omega$ such that $ g $ and $g_0$ induce the same metric $\gamma$ on $ \Sigma$.
Applying Theorem \ref{main-1-intro} to $(\R^3 \setminus \Omega, g_0) \cup (\Omega, g)$, one has
\begin{equation}
\begin{split}
 \frac{1}{8\pi} \int_\S ( H_0 - H_{_\Omega}) |\nabla u| \ge & \  \frac{1}{16 \pi} \left[ \int_{\R^3 \setminus \Omega} \frac{|\nabla^2 u|}{|\nabla u|}  
+ \int_{ \Omega  }\left( \frac{|\nabla^2 u_{_\Omega} |}{|\nabla u_{_\Omega} |}+R|\nabla u_{_\Omega} |\right)   \right]   .
\end{split}
\end{equation}
Here $u$ and $u_{_\Omega}$ are given by Theorem \ref{main-1-intro}. In particular, $u$ is harmonic on $ (\R^3 \setminus \Omega, g_0)$, 
asymptotic to a coordinate function at infinity, and satisfies $ \int_{\Sigma} \frac{\p u}{\p \nu}  = 0 $ at $\Sigma$.

If $ g $ has nonnegative scalar curvature, it follows that
\be \label{eq-ana-ST}
 \frac{1}{8\pi} \int_\S ( H_0 - H_{_\Omega} ) |\nabla u| \ge 0 . 
\ee
Comparing to the positivity of Brown-York mass (\cite{ST}), \eqref{eq-ana-ST} has the drawback  of having a weight function $ | \nabla u | $ 
that is not explicitly determined by the boundary data of $(\Sigma, \gamma, H_{_\Omega})$. 
On the other hand, \eqref{eq-ana-ST} does not require an assumption 
on the mean curvature $H_{_\Omega}$ and the Gauss curvature of $\gamma$.  

The rest of the paper is organized as follows. 
In Section \ref{sec-bdry-formulae},  we derive the relevant formulae on manifolds with boundary in the spirit of \cite{BKKS, S}.
In Section \ref{sec-mass-fill-in},  we study regularity of harmonic functions on manifolds with corner along a hypersurface and 
prove Theorem \ref{main-1-intro},  Corollary \ref{cor-WY} and \ref{Green-intro}.
In Section \ref{sec-BY-mass}, we analyze the convergence of Brown-York mass using Theorem \ref{main-1-intro} 
and prove Theorem \ref{main-2-intro}. 

\section{The Boundary Formulae} \label{sec-bdry-formulae}

We start with a lemma. 

\begin{lemma} \label{lem-bdry-g}
Let $ U  $ be an open set with piecewise smooth boundary $  \p U $ in a given manifold. 
Let $  S$ be a smooth portion of $\p U$. Let $ u $ be a harmonic function on $U $. Suppose $ u$ is $ C^2$ up to $S$.
Let $ \nu $ denote the outward unit normal to $ S$. Given a constant $ \epsilon > 0$, let 
$ \phi = \sqrt{ | \nabla u |^2 + \epsilon }$. Then, at $ p \in S$, 
\begin{enumerate}
\item $ \p_\nu \phi  = 0 $, if $ \nabla u (p) = 0$; 

\item  $ \p_\nu \phi =  - \frac{ | \nabla u |^2   }{\phi} H  -  \frac{ \p_\nu u  }{\phi}  \Delta_{_S} \eta $, 
if $ \nabla u (p) \neq 0$ and $ \nabla_{_S} \eta (p) = 0 $;

\item  $ \p_\nu \phi =  -   \frac{ | \nabla u |^2 }{\phi}  H  + \frac{ | \nabla u | }{\phi}   \kappa | \nabla_{_S } \eta |
+  \frac{ 1 }{\phi}  \la \nabla_{_S} \p_\nu u, \nabla_{_S} \eta  \ra  
- \frac{ \p_\nu u }{\phi}  | \nabla_{_S} \eta |^{-2} \,  \nabla^2_{_S} \eta ( \nabla_{_S} \eta  , \nabla_{_S} \eta ) $, 
if $ \nabla_{_S} \eta (p) \neq 0$.
\end{enumerate} 
Here $ \eta = u |_{_S} $ is the restriction of $u$ to $S$; $ \nabla_{_S}$, $ \nabla^2_{_S}$ and $\Delta_{_S} $  denote the gradient,  the Hessian and the Laplacian  on $S$; $ H$ is the mean curvature of $ S$ with respect to $\nu$; 
and $ \kappa$ is the geodesic curvature of the level curve  of $\eta$  in the level surface 
of $u$ with respect to the outward direction. 
\end{lemma}

\begin{proof}
At $p$, we  have
\be
\p_\nu \phi =  \frac{1}{ \phi}  \la \nabla _\nu \nabla u , \nabla u \ra =  \frac{1}{\phi} \nabla^2 u ( \nu, \nabla u ) .
\ee
Thus, $ \p_\nu \phi = 0 $ if $ \nabla u = 0 $ at $p$. 
In what follows, we assume $ \nabla u (p) \ne 0 $. Near $p$ in $S$, let 
$ n = \frac{\nabla u }{ | \nabla u | }$.
Then
\be
\begin{split}
\p_\nu \phi = & \ \frac{ | \nabla u | }{\phi} \,  \nabla^2 u \left( \nu,  n \right).
\end{split}
\ee

Let $ \theta \in [ 0, \pi ]$ be the angle between $ \nu $ and $ n$. 
We have two cases: 

\vspace{0.2cm}

\noindent {Case 1}. $ \nabla_{_S} \eta = 0 $ at $ p$. In this case, $ \nabla u = (\p_\nu u) \, \nu $, and
\be \label{eq-nu-phi-1}
\begin{split}
\p_\nu \phi 
=  \frac{ \p_\nu u  }{\phi} \,  \nabla^2 u \left( \nu, \nu \right) 
= \frac{ \p_\nu u  }{\phi} \,  (  - H \p_\nu u - \Delta_{_S} \eta ) . 
\end{split}
\ee
Here we have used the fact that $ \Delta u = 0 $ implies 
\be \label{eq-harmonic-bdry}
\nabla^2 u ( \nu, \nu) = - H \p_\nu u - \Delta_{_S} \eta . 
\ee
We rewrite \eqref{eq-nu-phi-1} as
\be \label{eq-nu-phi-2}
\begin{split}
\p_\nu \phi =  - \frac{ | \nabla u |^2   }{\phi} H  -  \frac{ \p_\nu u  }{\phi}  \Delta_{_S} \eta  . 
\end{split}
\ee

\vspace{0.2cm}

\noindent {\bf Case 2}.  $\nabla_\Sigma \eta \ne 0 $ at $p$.  In this case, we introduce the following notations near $p$:

\begin{itemize}

\item $ \tau_t   = \eta^{-1} (t)$ and $ \Sigma_t = u^{-1}(t)$; 

\item $ \tau_t'$ denotes a unit tangent vector to the curve $\tau_t $;

\item $ n_t = \frac{1}{ | \nabla_{_S} \eta  |} \nabla_{_S} \eta $, which is a unit normal to $\tau_t$ in $S$;

\item $ \nu_{t} $ denotes the outward unit normal to $ \tau_t $ within $ \Sigma_t$.

\end{itemize}

Along the curve $ \tau_t$, 
$ \{ \nu, n_t  \} \ \text{and} \ \{ \nu_t , n \} $
both form an orthonormal basis of the normal bundle $ {\tau_t'}^\perp$,
and
\be \label{eq-theta-1-2}
\cos \theta =  \la \nu, n \ra  = \frac{ \p_\nu u  }{| \nabla u |}  \ \ 
\text{and} \ \ 
 | \nabla_{_S} \eta | = | \nabla u |  \sin \theta . 
\ee
We have
\be
\begin{split}
\p_\nu \phi = & \ \frac{ | \nabla u | }{\phi} \,  \nabla^2 u \left( \nu,  n \right) \\
= & \ \frac{ | \nabla u | }{\phi} \,  \nabla^2 u ( \nu, n_t \sin \theta  +  \nu  \cos \theta) \\
= & \  \frac{ | \nabla u | }{\phi} \, \sin \theta \, \nabla^2 u ( \nu, n_t )  + \frac{ | \nabla u | }{\phi} \, \cos \theta \, \nabla^2 u ( \nu, \nu)  .
\end{split}
\ee
Let $ \Pi (v, w) = \la \nabla_v \nu, w \ra $ denote the second fundamental form of $S$ with respect to $\nu$. Then
\be
\begin{split}
\nabla^2 u ( \nu, n_t) = &  \ 
\la \nabla_{n_t} \nabla u , \nu ) \\
= & \ \la \nabla_{n_t} \nabla_{_S} \eta , \nu) + \la \nabla_{n_t} (\p_\nu u) \nu , \nu \ra \\
= & \ - \Pi ( \nabla_{_S} \eta , n_t) + n_t ( \p_\nu u ) .
\end{split}
\ee
Therefore,
\be
\sin \theta \, \nabla^2 u ( \nu, n_t )  
 =  - \Pi ( n_t , n_t) | \nabla_{_S } \eta | \sin \theta + n_t ( \p_\nu u ) \sin  \theta .
\ee
Thus,
\be \label{eq-nu-phi-3}
\begin{split}
 \p_\nu \phi = & \  \frac{ | \nabla u | }{\phi} \, \left[ - \Pi ( n_t , n_t) | \nabla_{_S } \eta | \sin \theta + n_t ( \p_\nu u ) \sin  \theta \right. \\
 & \ \left. + \left(  - H \p_\nu u - \Delta_{_S} \eta \right) \cos \theta \right] . 
\end{split}
\ee
Here we also have used \eqref{eq-harmonic-bdry}.

\bigskip

Now let $\kappa$ denote the geodesic curvature of $\tau_t = \Sigma_t \cap S $ in $\Sigma_t$ with respect to $\nu_t$, 
that is
\be
\kappa =  \la \nabla_{ \tau_t'} \nu_t , \tau_t ' \ra = - \la \nabla_{ \tau_t'} \tau'_t , \nu_t , \ra 
\ee
Using the fact  $ \nu_t = - \cos \theta \, n_t + \sin \theta \, \nu $, 
we have
\be
\begin{split}
 \kappa 
= & \ - \la \nabla_{\tau_t ' } \tau_t' , - \cos \theta \, n_t + \sin \theta \, \nu \ra \\
= & \  \sin \theta \, \Pi (\tau_t', \tau_t' ) + \cos \theta \, \la \nabla_{\tau_t'} \tau_t', n_t \ra \\ 
= & \  \sin \theta \, \left[ H - \Pi (n_t, n_t) \right]  + \cos \theta \, \la \nabla_{\tau'} \tau', n_t \ra .
\end{split}
\ee
Or equivalently, 
\be \label{eq-kappa-H}
- \Pi (n_t, n_t) \sin \theta =  \kappa - \sin \theta \, H  -  \cos \theta \, \la \nabla_{\tau'} \tau', n_t \ra .
\ee
It follows from \eqref{eq-nu-phi-3}, \eqref{eq-kappa-H} and \eqref{eq-theta-1-2} that
\be \label{eq-nu-phi-4}
\begin{split}
\p_\nu \phi = & \  \frac{ | \nabla u | }{\phi} \, \left[ \left( \kappa - \sin \theta \, H  -  \cos \theta \, \la \nabla_{\tau'} \tau', n_t \ra \right) 
 | \nabla_{_S } \eta | + n_t ( \p_\nu u ) \sin  \theta \right. \\
 & \ \left. + \left(  - H \p_\nu u - \Delta_{_S} \eta \right) \cos \theta \right] \\
= & \  \frac{ | \nabla u | }{\phi} \, \left[  \kappa | \nabla_{_S } \eta | - H | \nabla u | + n_t ( \p_\nu u ) \sin  \theta \right. \\
 & \ \left. - \left( \Delta_{_S} \eta +  \la \nabla_{\tau'} \tau', \nabla_{_S} \eta  \ra  \right) \cos \theta \right] .
\end{split}
\ee
Observe that
\be
| \nabla u |  n_t ( \p_\nu u ) \sin  \theta = \la \nabla_{_S} \p_\nu u, \nabla_{_S} \eta  \ra 
\ee
and
\be
\Delta_{_S} \eta +  \la \nabla_{\tau'} \tau', \nabla_{_S} \eta  \ra = \nabla^2_{_S} \eta ( n_t , n_t) . 
\ee
We can rewrite \eqref{eq-nu-phi-4} as
\be
\p_\nu \phi =   \frac{ | \nabla u | }{\phi}   \kappa | \nabla_{_S } \eta | -   \frac{ | \nabla u |^2 }{\phi}  H 
+  \frac{ 1 }{\phi}  \la \nabla_{_S} \p_\nu u, \nabla_{_S} \eta  \ra  
- \frac{ \p_\nu u }{\phi}  \nabla^2_{_S} \eta ( n_t , n_t)   .
\ee
This finishes the proof.
\end{proof}

The next proposition gives an analogue of Proposition 4.2 in \cite{BKKS} with no boundary condition imposed.

\begin{proposition} \label{prop-bdry-1}
Let $ (\Omega, g)$ be a $3$-dimensional compact, oriented Riemannian manifold with smooth boundary $\p \Omega$.
Here $ \p \Omega $ is not necessarily connected. Let $ u $ be a harmonic function on $\Omega$. 
Then
\be \label{eq-prop-bdry}
\begin{split}
& \ \int_{\Omega}  \frac{1}{2} \left[ \frac{ | \nabla^2 u |^2 }{  | \nabla u |  }  +  R | \nabla u |   \right] 
  -  2 \pi \int_{u_{min} }^{ u_{max} } \chi(\Sigma_t)  \, dt   \\
\le & \ -  \int_{\p \Omega} H | \nabla u |   
  +  \int_{ \{  \nabla_{_{\p \Omega} } u \ne 0 \} }  \left \la \nabla_{_{\p \Omega}} \p_\nu u, \frac{ \nabla_{_{\p \Omega} } u  }{ | \nabla u |} \right \ra  
  - \frac{ \p_\nu u }{ | \nabla u | }    \nabla^2_{_{ \p \Omega } } u \left( \frac{ \nabla_{_{ \p \Omega} } u }{ | \nabla_{_{\p \Omega}} u | }   , \frac{ \nabla_{_{ \p \Omega} } u }{ | \nabla_{_{\p \Omega}} u | }  \right) .
\end{split}
\ee
Here $R$ is the scalar curvature of $g$; $\Sigma_t = u^{-1}(t)$ is the level set of $u$; $\chi (\Sigma_t)$ denotes the Euler characteristic of $\Sigma_t$ if $\Sigma_t$ is a regular surface;
$u_{min} = \min_\Omega u $, $ u_{max} = \max_\Omega u$; and 
$H$ is the mean curvature of $ \p \Omega$ with respect to the outward unit normal $\nu$.
\end{proposition}

\begin{proof}
Given $\epsilon > 0$, let $ \phi = \sqrt{ | \nabla u |^2 + \epsilon}$. 
Integrating $\Delta \phi$ on $\Omega$ and applying Lemma \ref{lem-bdry-g}, we have
\be
\int_\Omega \Delta \phi =  \int_{\p \Omega} \p_\nu \phi 
=  \int_{\p \Omega \cap \{ \nabla u \neq 0 \} } \p_\nu \phi 
=  I_1 + I_2 ,
\ee
where 
\be
I_1 =  \int_{\{  \nabla_{_{\p \Omega}} u  = 0 , \ \p_\nu u \ne 0 \} } \p_\nu \phi
=  \int_{\{  \nabla_{_{\p \Omega}} u  = 0 , \ \p_\nu u \ne 0 \} } - \frac{ | \nabla u |^2   }{\phi} H  -  \frac{ \p_\nu u  }{\phi}  \Delta_{_{\p \Omega} } u ,
\ee
and
\be
\begin{split}
I_2 = & \  \int_{ \{  \nabla_{_{\p \Omega} } u \ne 0 \} } \p_\nu \phi \\
= & \  \int_{ \{  \nabla_{_{\p \Omega} } u \ne 0 \} }
 -   \frac{ | \nabla u |^2 }{\phi}  H  + \frac{ | \nabla u | }{\phi}   \kappa | \nabla_{_{\p \Omega }} u |  +  \frac{ 1 }{\phi}  \la \nabla_{_{\p \Omega}} \p_\nu u, \nabla_{_{\p \Omega} } u  \ra  \\
& \ \ \ \ \ \ \ \ \ \ \ \ \ \ \ - \frac{ \p_\nu u }{\phi}  | \nabla_{_{\p \Omega}} u |^{-2} \,  \nabla^2_{_{ \p \Omega } } u ( \nabla_{_{ \p \Omega} } u  , \nabla_{_{\p \Omega} } u ) .
\end{split}
\ee
Here $ | \nabla u |$, $H$, $\Delta_{_{\p \Omega} } u $, $ | \nabla_{_{\p \Omega} } \p_\nu u |$, and $ | \nabla^2_{_{\p \Omega} } u |$ are all
 bounded quantities on $ \p \Omega$, independent on $\epsilon$.  Moreover, by  item 3 in Lemma \ref{lem-bdry-g}, 
\be
 \frac{ | \nabla u | }{\phi}   \kappa | \nabla_{_{\p \Omega}  } u |
 =  \p_\nu \phi +   \frac{ | \nabla u |^2 }{\phi}  H  -  \frac{ 1 }{\phi}  \la \nabla_{_{\p \Omega} } \p_\nu u, \nabla_{_{\p \Omega} } u  \ra  
+  \frac{ \p_\nu u }{\phi}  | \nabla_{_{\p \Omega} } u |^{-2} \,  \nabla^2_{_{\p \Omega} } u ( \nabla_{_{\p \Omega} } u  , \nabla_{_{\p \Omega}} u ) 
\ee
on $ \{ \nabla_{_{\p \Omega} } u \neq 0 \} \subset \p \Omega$.
This, together with  $ | \p_\nu \phi | \le | \nabla \phi | \le | \nabla^2 u | $, shows $ \frac{ | \nabla u | }{\phi}   \kappa | \nabla_{_{\p \Omega} } u | $ 
is bounded by a constant independent on $\epsilon$ as well. 
Therefore, by dominated convergence theorem, 
\be
\begin{split}
\lim_{\epsilon \to 0}  \int_{\p \Omega} \p_\nu \phi = & \ 
\int_{\{  \nabla_{_{\p \Omega}} u  = 0 , \ \p_\nu u \ne 0 \} } - H | \nabla u |   -  \frac{ \p_\nu u  }{|\nabla u |}  \Delta_{_{\p \Omega} } u \\
 & \ +  \int_{ \{  \nabla_{_{\p \Omega} } u \ne 0 \} }
 -  H | \nabla u |  +   \kappa | \nabla_{_{\p \Omega }} u |  +  \frac{ 1 }{| \nabla u | }  \la \nabla_{_{\p \Omega}} \p_\nu u, \nabla_{_{\p \Omega} } u  \ra  \\
& \ \ \ \ \ \ \ \ \ \ \ \ \ \ \ \ \  - \frac{ \p_\nu u }{ | \nabla u | }  | \nabla_{_{\p \Omega}} u |^{-2} \,  \nabla^2_{_{ \p \Omega } } u ( \nabla_{_{ \p \Omega} } u  , \nabla_{_{\p \Omega} } u ) .
\end{split}
\ee
Combining the terms involving $H$ and applying the coarea formula to the term involving $\kappa$, we can rewrite the above as
\be \label{eq-bdry-gg-0}
\begin{split}
\lim_{\epsilon \to 0}  \int_{\p \Omega} \p_\nu \phi = & \ - \int_{\p \Omega} H | \nabla u | 
+ \int_{u_{min}}^{u_{max} } \left( \int_{ \{  \nabla_{_{\p \Omega} } u \ne 0 \}\cap \Sigma_t } \kappa \right) \, d t \\
& \  + \int_{\{  \nabla_{_{\p \Omega}} u  = 0 , \ \p_\nu u \ne 0 \} }  -  \frac{ \p_\nu u  }{|\nabla u |}  \Delta_{_{\p \Omega} } u \\
 & \ +  \int_{ \{  \nabla_{_{\p \Omega} } u \ne 0 \} } \la \nabla_{_{\p \Omega}} \p_\nu u, \frac{ \nabla_{_{\p \Omega} } u  }{ | \nabla u |}  \ra  
  - \frac{ \p_\nu u }{ | \nabla u | }    \nabla^2_{_{ \p \Omega } } u ( \frac{ \nabla_{_{ \p \Omega} } u }{ | \nabla_{_{\p \Omega}} u | }   , \frac{ \nabla_{_{ \p \Omega} } u }{ | \nabla_{_{\p \Omega}} u | }  ) .
\end{split}
\ee
Under an assumption $u$ is $C^2$ on $\Sigma$, 
elementary arguments shows the set 
$$ \{ x \in \Sigma \ | \  \nabla_{_{\p \Omega}} u  = 0 , \, \Delta_{_{\p \Omega} } u \neq 0 \} $$
is a set of measure zero in $\Sigma$. 
Hence, the integral in the middle line above vanishes and 
\eqref{eq-bdry-gg-0} simplifies to 
\be \label{eq-bdry-gg}
\begin{split}
\lim_{\epsilon \to 0}  \int_{\p \Omega} \p_\nu \phi = & \ - \int_{\p \Omega} H | \nabla u | 
+ \int_{u_{min}}^{u_{max} } \left( \int_{ \{  \nabla_{_{\p \Omega} } u \ne 0 \}\cap \Sigma_t } \kappa \right) \, d t \\
 & \ +  \int_{ \{  \nabla_{_{\p \Omega} } u \ne 0 \} } \la \nabla_{_{\p \Omega}} \p_\nu u, \frac{ \nabla_{_{\p \Omega} } u  }{ | \nabla u |}  \ra  
  - \frac{ \p_\nu u }{ | \nabla u | }    \nabla^2_{_{ \p \Omega } } u ( \frac{ \nabla_{_{ \p \Omega} } u }{ | \nabla_{_{\p \Omega}} u | }   , \frac{ \nabla_{_{ \p \Omega} } u }{ | \nabla_{_{\p \Omega}} u | }  ) .
\end{split}
\ee

Next, we analyze the interior terms following the argument of Theorem 1.1 in \cite{S}.
First, by Bochner's  formula, 
\be
\begin{split}
\Delta \phi = & \ \frac{1}{\phi} \left[ | \nabla^2 u |^2 + \Ric (\nabla u, \nabla u) - \frac{1}{4 \phi^2} | \nabla | \nabla u |^2 |^2 \right] \\
\ge & \ \phi^{-1} \Ric (\nabla u , \nabla u ) .
\end{split}
\ee

Let $ \mathcal{C} $ be the union of critical values of both $u: \Omega \to \R$ and $ u |_{_{\p \Omega}}: \p \Omega \to \R $. By Sard's theorem, $\mathcal{C}$ 
is a set of measure zero. 
Let $ A $ be an open set of $[ u_{min}, u_{max}]$ 
so that $ A$ contains $\mathcal{C}$.
Let $ B $ be the complement of $ A$ in $ [ u_{min}, u_{max} ] $.
By formula (14) in \cite{S}, along a regular level set $\Sigma_t$ of $u$,
\be
\Delta \phi \ge \frac{1}{2 \phi} \left[ | \nabla^2 u |^2 + | \nabla u |^2 ( R - 2 K_{_\Sigma} ) \right] ,
\ee
where $ K_{_\Sigma}$ is  the Gauss  curvature of $\Sigma_t$.
Thus,
\be \label{eq-interior-1}
\begin{split}
 \int_\Omega \Delta \phi\ge & \ \int_{u^{-1}(B)}  \frac{1}{2 \phi} \left[ | \nabla^2 u |^2 + | \nabla u |^2 ( R - 2 K_{_\Sigma} ) \right] \\
& \   -  \max_{\Omega} | \Ric |  \int_{u^{-1}(A) } | \nabla u | \, d t .
\end{split}
\ee
As $ K_{_\Sigma}$ is bounded by a constant independent on $\epsilon$ on the compact set $ u^{-1} (B)$, passing  to the limit 
in \eqref{eq-interior-1} gives
\be \label{eq-interior-2}
\begin{split}
\lim_{\epsilon \to 0} \int_{\Omega} \Delta \phi \ge  & \ \int_{u^{-1}(B)}  \frac{1}{2 | \nabla u | } \left[ | \nabla^2 u |^2 + | \nabla u |^2 ( R - 2 K_{_\Sigma} ) \right] \\
& \   -  \max_{\Omega} | \Ric |  \int_{u^{-1}(A) } | \nabla u | \, d t .
\end{split}
\ee
By coarea formula, this can be written as
\be \label{eq-interior-2-2}
\begin{split}
\lim_{\epsilon \to 0} \int_{\Omega} \Delta \phi \ge & \  \int_{B} \left( \int_{\Sigma_t}  \frac{1}{2 | \nabla u |^2 } \left[ | \nabla^2 u |^2 + | \nabla u |^2 R  \right] \right) \, dt  \\
 & \  - \int_{B}  \left(  \int_{\Sigma_t} K_{_\Sigma}  \right) \, d t  - \left( \max_{\Omega} | \Ric | \right)  \int_{A} | \Sigma_t | \, d t.
\end{split}
\ee
As $ t \in B$ is a regular value for both $ u $ and $ u|_{_{\p \Omega}}$,  $ \Sigma_t$ is a compact surface with boundary satisfying
$$ \p \Sigma_t = \Sigma_t \cap \p \Omega = \Sigma_t \cap \{ \nabla_{_{\p \Omega} } u \ne 0 \} . $$
By the Gauss-Bonnet theorem, 
\be \label{eq-interior-3}
\begin{split}
\lim_{\epsilon \to 0} \int_{ \Omega} \Delta  \phi \ge & \  \int_{B} \left( \int_{\Sigma_t}  \frac{1}{2 | \nabla u |^2 } \left[ | \nabla^2 u |^2 + | \nabla u |^2 R  \right] \right) \, dt  \\
 & \  - \int_{B}  \left(  2 \pi \chi(\Sigma_t) - \int_{\Sigma_t \cap \{ \nabla_{_{\p \Omega} } u \ne 0 \}} \kappa  \right) \, d t  - \left( \max_{\Omega} | \Ric | \right)  \int_{A} | \Sigma_t | \, d t.
\end{split}
\ee
Since  $ \int_{ u_{min} }^{{u_{max} } }  | \Sigma_t | \, d t =  \int_{\Omega} | \nabla u |  < \infty$,  \eqref{eq-prop-bdry} now follows from \eqref{eq-bdry-gg} and  \eqref{eq-interior-3} by letting the measure of $ A$ tend to zero. 
\end{proof}

\begin{remark} \label{rem-boundary-angle}
We learn from Daniel Stern that, at points $ \nabla_{_{\p \Omega} } u \neq 0 $, 
the integrand in the second boundary integral in  \eqref{eq-prop-bdry} can be written as
\be \label{eq-angle-Stern}
\begin{split}
& \ \left \la \nabla_{_{\p \Omega}} \p_\nu u, \frac{ \nabla_{_{\p \Omega} } u  }{ | \nabla u |} \right \ra  
  - \frac{ \p_\nu u }{ | \nabla u | }    \nabla^2_{_{ \p \Omega } } u \left( \frac{ \nabla_{_{ \p \Omega} } u }{ | \nabla_{_{\p \Omega}} u | }   , \frac{ \nabla_{_{ \p \Omega} } u }{ | \nabla_{_{\p \Omega}} u | }  \right) \\
= & \ \left \la \frac{ | \nabla_{_{\p \Omega} } u | \nabla_{_{\p \Omega} } \p_\nu u - \p_\nu u \nabla_{_{\p \Omega} } | \nabla_{_{\p \Omega}}  u | }
{ | \nabla u |^2} , \frac{ \nabla_{_{\p \Omega}} u }{ | \nabla_{_{\p \Omega}} u | } \right \ra  | \nabla u | \\   
= & \ \left \la \nabla_{_{\p \Omega}} \beta , \frac{ \nabla_{_{\p \Omega}} u }{ | \nabla_{_{\p \Omega}} u |}  \right \ra | \nabla u |,
\end{split}
\ee
where $ \displaystyle  \beta = \frac{\pi}{2} - \theta $ is the angle between $\nabla u $ and the boundary surface $\p \Omega$. 
It is intriguing  if this interpretation can yield relevant geometric boundary conditions for harmonic functions. 
\end{remark}

Next we consider the mass of an asymptotically flat $3$-manifold with boundary.
Adopting conventions from \cite{Bartnik}, we say 
a Riemannian manifold $(M^3, g)$ is asymptotically flat if there is a compact set $K \subset M$
and a diffeomorphism $ \Psi: M \setminus K \rightarrow \R^3 \setminus B_r  $, where $ B_r = \{ | x | \le r \}$ for some $ r $,
such that 
\be \label{eq-g-AF}
(\Psi^{-1})^*(g)_{ij} - \delta_{ij} \in W^{k,p}_{-\tau} ( \R^3 \setminus B_r  ),
\ee
where  $ k \ge 2 $, $ p > 3$ and $ \tau \in (\frac12 , 1] $.
Here $ W^{k,p}_\delta ( \R^3 \setminus B_r )$ denotes the weighted Sobolev space 
consisting of functions $ u \in W^{k,p}_{loc} ( \R^3 \setminus B_r )$ such that
\be
|| u ||_{W^{k,p}_{\delta} ( \R^3 \setminus B_r ) } =\sum_{|\beta|\leq k} 
||  \partial^\beta u \, |x| ^{|\beta| - \d -\frac{3}{p} } ||_{L^p (\R^3\setminus B_r )} < \infty. 
\ee
If the scalar curvature $R$ of $g$ is integrable,   the ADM mass \cite{ADM} of $(M,g)$ 
is defined and given by
\be
\m = \lim_{r \to \infty} \frac{1}{16 \pi}  \int_{S_r} (g_{ij,i} - g_{ii,j} ) v^j .
\ee
Here $ S_r = \{ | x | = r \} $ and $ v$ denotes the outward unit normal to $S_r$. 
The fact that $\m$ is a geometric invariant of $(M, g)$, independent on the choice of the coordinates,
was shown by Bartnik \cite{Bartnik} and by Chru\'{s}ciel \cite{C} independently.

\begin{proposition} \label{prop-mass-bdry}
Let $(M^3, g)$ be an asymptotically flat $3$-manifold with boundary $\Sigma$. 
Given  a function $\eta$ on $\Sigma$, let $u $ be the harmonic function on $M$ such that $ u = \eta $ on $ \Sigma$ and 
$ u $ is asymptotic to some coordinate function $x_i$ at the infinity.  Let $ \Sigma_t = u^{-1}(t)$ be the level set of $u$.
Let $\m $ be the mass of $(M, g)$ and  $\mu $ be the infinity pointing 
unit normal to $\Sigma$. Then

\be \label{eq-mass-bdry}
\begin{split}
& \ 8 \pi \m +  2 \pi  \int_{- T }^{T}  ( \chi(\Sigma_t) - 1 )  \, dt \\
\ge & \ \int_{M}  \frac{1}{2} \left[ \frac{ | \nabla^2 u |^2 }{  | \nabla u |  }  +  R | \nabla u |   \right] 
 - \int_{\Sigma} H | \nabla u |  \\
 & \ +  \int_{ \{  \nabla_{_{\Sigma} } \eta \ne 0 \} } \left \la \nabla_{_{\Sigma}} \p_\mu u, \frac{ \nabla_{_{\Sigma} } \eta  }{ | \nabla u |}  \right \ra  
  - \frac{ \p_\mu u }{ | \nabla u | }    \nabla^2_{_{ \Sigma } } \eta  \left( \frac{ \nabla_{_{ \Sigma} } \eta  }{ | \nabla_{_{\Sigma}} \eta | } , 
  \frac{ \nabla_{_{ \Sigma} } \eta  }{ | \nabla_{_{\Sigma}} \eta | }  \right) .
\end{split}
\ee
Here $T >0 $ is some constant satisfying $ \chi(\Sigma_t) = 1 $ for all $ | t | \ge T$, 
$ R$ is the scalar curvature of $g$, and $H$ is the mean curvature of $\Sigma$ with respect to $\mu$.
\end{proposition}

Formula \eqref{eq-mass-bdry} is a corollary of the result 
of Bray-Kazaras-Khuri-Stern (Section 6 of \cite{BKKS}) and the preceding Proposition \ref{prop-bdry-1}.  

\begin{proof} 
Suppose $u$ is asymptotic to $ x_1$. 
By constructing two other harmonic functions $u_2$ and $u_3$, asymptotic to $x_2$ and $x_3$, respectively, 
we may well assume $ u = y_1$ outside a compact set $K$, where  $ \{ y_1, y_2, y_3 \}$ are harmonic coordinates 
on $M \setminus K$ (see \cite{Bartnik} for instance).
This, in particular, implies that there exists a  constant $T> 0 $ such that $\chi (\Sigma_t) = 1 $ 
for all $ | t | \ge T_0 $.
For instance,  let $ D \subset M$ be a compact set such that $ M \setminus D=
\R^3 \setminus \{ (y_1, y_2, y_3) \ | \ | y_i | \le a, \, i = 1, 2, 3 \}$ for some constant $ a > 0$. 
Then,  for any $ t $ with $ |t | >  \max \{a, \max_{D} |u| \} $,   $ \Sigma_t$ is precisely 
the entire coordinate plane $ \{ y_1 = t \}$, hence $ \chi (\Sigma_t) =  1$ for such $t$.

For each large constant $ L > 0$, let $ S$ denote the cylindrical surface given by
$$ S = \left\{ y_1 = \pm L , \, y_2^2 + y_3^2 \le L  \right\} \cup \left\{  | y _1 | \le  L, \, y_2^2 + y_3^2 = L \right\} .$$
Let $ \Omega $ be the region bounded by $ S$ and $ \Sigma$. Let $ \phi = \sqrt{ | \nabla u |^2 + \epsilon }$ for $\epsilon > 0$.
Though $ \Omega $ does not have a smooth boundary, Stoke's theorem still holds on this $\Omega$. Thus,
\be \label{eq-u-u-ep}
\begin{split}
\int_S \p_\mu | \nabla u | =  \lim_{\epsilon \to 0} \int_{S} \p_\mu \phi = \lim_{\epsilon \to 0}  \int_{\Sigma} \p_\mu \phi + \lim_{\epsilon \to 0}  \int_\Omega \Delta \phi .
\end{split}
\ee
Here, by abuse of notations, $\mu$ also denotes the infinity pointing unit normal to $S$ and we use the fact $ | \nabla u | > 0 $ at $ S$ for  large $L$.
For the other boundary term $ \lim_{\epsilon \to 0}  \int_{\Sigma} \p_\mu \phi $ and the interior term 
$  \lim_{\epsilon \to 0}  \int_\Omega \Delta \phi $, we can apply \eqref{eq-bdry-gg} and \eqref{eq-interior-3}, respectively, to have
\be \label{eq-u-u-ep-1}
\begin{split}
& \ \lim_{\epsilon \to 0}  \int_{\Sigma} \p_\mu \phi + \lim_{\epsilon \to 0}  \int_\Omega \Delta \phi \\
\ge & \ 
\int_\Omega   \frac{1}{2 } \left[  \frac{ | \nabla^2 u |^2 } { | \nabla u |}  
+ R | \nabla u | \right]  \, dt   - \int_{\min_\Omega u }^{\max_\Omega u}   2 \pi \chi(\Sigma_t)  \, dt  + \int_{\min_\Omega u }^{\max_\Omega u} 
  \int_{\Sigma_t \cap S } \kappa  \, d t \\
& \   - \int_{\Sigma } H | \nabla u | 
-  \int_{\{  \nabla_{_{\Sigma }} u  = 0 , \ \p_\mu u \ne 0 \} } \frac{ \p_\mu u  }{|\nabla u |}  \Delta_{_{\Sigma} } u \\
 & \ +  \int_{ \{  \nabla_{_{\Sigma } } u \ne 0 \} } \left \la \nabla_{_{\Sigma}} \p_\mu u, \frac{ \nabla_{_{\Sigma} } u  }{ | \nabla u |}  \right \ra  
  - \frac{ \p_\mu u }{ | \nabla u | }    \nabla^2_{_{ \Sigma } } u \left( \frac{ \nabla_{_{ \Sigma} } u }{ | \nabla_{_{\Sigma}} u | }   , \frac{ \nabla_{_{ \Sigma} } u }{ | \nabla_{_{\Sigma}} u | } \right ) .
\end{split}
\ee
For large $L$,  observe that $ \max_{\Omega} u = \max_{S} u = L $ and $ \min_{\Omega} u = \min_{S} u = - L $.
Hence, by \eqref{eq-u-u-ep} and \eqref{eq-u-u-ep-1}, 
\be \label{eq-u-u-ep-2}
\begin{split}
& \ \int_{S} \p_\mu | \nabla u | +   \int_{-L }^{L}   \left( 2 \pi \chi(\Sigma_t) -  \int_{\Sigma_t \cap S } \kappa  \right) \, d t \\
\ge & \ 
\int_\Omega   \frac{1}{2 } \left[  \frac{ | \nabla^2 u |^2 } { | \nabla u |}  
+ R | \nabla u | \right]  \, dt     - \int_{\Sigma } H | \nabla u |  -  \int_{\{  \nabla_{_{\Sigma }} u  = 0 , \ \p_\mu u \ne 0 \} } \frac{ \p_\mu u  }{|\nabla u |}  \Delta_{_{\Sigma} } u \\
 & \ +  \int_{ \{  \nabla_{_{\Sigma } } u \ne 0 \} } \left \la \nabla_{_{\Sigma}} \p_\mu u, \frac{ \nabla_{_{\Sigma} } u  }{ | \nabla u |}  \right \ra  
  - \frac{ \p_\mu u }{ | \nabla u | }    \nabla^2_{_{ \Sigma } } u  \left ( \frac{ \nabla_{_{ \Sigma} } u }{ | \nabla_{_{\Sigma}} u | }   , \frac{ \nabla_{_{ \Sigma} } u }{ | \nabla_{_{\Sigma}} u | } \right ) .
\end{split}
\ee
Now recall the mass formula derived in (6.27) of \cite{BKKS}, which asserts 
\be \label{eq-mass-BKKS}
\lim_{ L \to \infty} \left[  \int_{S} \p_\mu | \nabla u | +   \int_{-L }^{L}   \left( 2 \pi -  \int_{\Sigma_t \cap S } \kappa  \right) \, d t \right] = 8 \pi \m .
\ee
Leting $ L \to \infty$,  \eqref{eq-mass-bdry} follows from \eqref{eq-u-u-ep-2} and \eqref{eq-mass-BKKS}.
\end{proof}

\section{Mass of manifolds with a fill-in of the boundary} \label{sec-mass-fill-in}

In this section, we use $(E, g)$, $(\Omega, g)$ to  denote an asymptotically flat $3$-manifold with boundary $ \Sigma$, 
 a compact Riemannian $3$-manifold with boundary $\p \Omega$,  respectively. 
 By abuse of notation,  metrics on $E$ and $\Omega$ are both denoted by $g$.
We say $ (\Omega, g)$ is a fill-in of $\Sigma$ if $ \p \Omega $ is isometric to $\Sigma$,
both of which are equipped with the induced metric. 
By a recent result of Shi, Wang and Wei \cite{SWW}, which answered a question of Gromov \cite{Gromov},
a fill-in of $\Sigma$ with nonnegative scalar curvature always exists.

Suppose  $(\Omega, g)$ is a fill-in of $(E, g)$. Let $U_+$ and $U_-$ 
be the Gaussian tubular neighborhoods  of $\Sigma $ in $E$ and $\Omega$, respectively, 
so that $ U_+$ is diffeomorphic to  $ \Sigma \times [0, \epsilon)$ and $U_-$ is diffeomorphic to 
$\Sigma \times (-\epsilon , 0]$. Identifying  $U_+ \cup U_- $ with $\Sigma \times (-\epsilon, \epsilon )$ by lining up geodesics 
that are normal to $\Sigma$, 
we obtain a manifold $(M, g)$ on which $g$ is a continuous metric, which is smooth up to $\Sigma$ from both sides.
In particular, $g$ is Lipschitz across $\Sigma$.  As a result, by \eqref{eq-g-AF}, 
$g$ is a $W^{1, p}_{loc} $ metric on $M$ satisfying 
$
 g_{ij} - \delta_{ij} \in W^{1, p}_{ - \tau} ( \R^3 \setminus B_r ) 
$
for $ p > 3$ at the end of $E$.
By Theorem 3.1 in \cite{Bartnik}, there exists a function $u \in W^{2,p}_{loc}(M)$ such that $u-x_1\in W^{2,p}_{1-\tau}( \R^3 \setminus B_r)$ and 
$ \Delta u = 0 $. 

We want to apply Proposition \ref{prop-bdry-1} and Propositon \ref{prop-mass-bdry} to analyze the mass of $(M, g)$
with this choice of $u$.

\subsection{Regularity of harmonic functions on the corner surface $\Sigma$} \label{Existence}

By Sobolev embedding, $u\in W^{2,p}_{loc} (M) \subset C^{1,\alpha}_{loc}(M)$, and hence $u |_{\Sigma} \in C^{1,\alpha}(\S)$. Standard elliptic theory shows $u$ is smooth away from  $\S$. In what follows, we examine the regularity of $u $ on $ \Sigma$. 

\begin{proposition} \label{prop-u-corner}
If $g$ is $C^2$ up to $\Sigma$ on its two sides in $M$, then $u |_\Sigma$, the restriction of $u$ to $ \Sigma$, is in $C^{2, \alpha} (\Sigma) $, and
$u $ is $C^{2,\alpha}$ up to $\Sigma$ on its two sides. Similarly, if $g$ is $C^\infty$ up to $\Sigma$, then $u |_{\Sigma}$ is in $C^\infty (\Sigma) $ and $u$ is $C^\infty $ up to $\Sigma$ on its two sides.
\end{proposition}

\begin{proof}
Consider a point $p\in \S$.  
Using Fermi coordinates, we identify a neighborhood $V$ of $p$ in $M$ 
with $  \S \times (-\eps,\eps)$, where coordinates on $\Sigma$ are $(x_1, x_2)$ and 
a coordinate in $(-\epsilon, \epsilon)$ is  $x_3 = t $.
Given a small constant $h$, consider the difference quotient of $u$ along the direction of $e_\alpha = \p_\alpha$, that is 
\be
\delta^h_\alpha u (x) = \frac{1}{h} [ u (x + h e_\alpha) - u (x) ], \ \alpha \in \{ 1, 2 \}. 
\ee
It follows from $ \Delta u = 0$ that
\be \label{eq-Delta-h-u}
\begin{split}
\Delta \,  \delta^h_\alpha u  (x) = & \ \frac{1}{h} \left[ g^{ij}(x)  \p^2_{ij} u ( x + h e_\alpha) - g^{ij}(x)  \Gamma^k_{ij} (x) \p_k u (x+h e_\alpha ) \right] \\
= & \ \frac{1}{h} \left[ [ g^{ij}(x) - g^{ij} (x+ h e_\alpha) ] \p^2_{ij} u ( x + h e_\alpha) \right. \\
& \ \left. - [ g^{ij}(x)  \Gamma^k_{ij} (x)  - g^{ij}(x + h e_\alpha )  \Gamma^k_{ij} ( x + h e_\alpha )  ] \p_k u (x+h e_\alpha ) \right] .
\end{split}
\ee
Since $ g_{ij} \in C^{0,1}_{loc} (V)$, 
\be
\frac{g^{ij}(x) - g^{ij} (x+ h e_\alpha)}{h},  \,  \Gamma^k_{ij} (x),  \, \Gamma^k_{ij} (x + h e_\alpha) \,  \in \,  L^\infty_{loc} (V). 
\ee
Since $g$ is $C^2$ up to $\Sigma$ from both sides of $\Sigma$, $g_{ij}$ and $ \p g_{ij}$ are locally Lipschitz 
along the $\p_\alpha$-direction in $V$,  therefore
\be
\frac{1}{h} \left[ \Gamma^{k}_{ij} (x) - \Gamma^{k}_{ij} ( x + h e_\alpha) \right] \, \in \, L^\infty_{loc} (V) . 
\ee
These show that the right side of \eqref{eq-Delta-h-u} is in $L_{loc}^p (V)$. 
By Theorem 9.11 in \cite{GT},  $|| \delta^h_\alpha u ||_{W^{2,p}(\tilde V)}$ is bounded by a constant independent on $h$,  
on any subdomain $ \tilde V \subset \subset V$. This readily implies $ \p_\alpha u = \lim_{h \to 0} \delta^h_\alpha u \in W^{2,p}_{loc} (V)$.
By Sobolev embedding, $ \p_\alpha u \in C^{1 , \alpha}_{loc} (V) $. In particular, this shows $u |_\Sigma  \in C^{2,\alpha} (\Sigma) $.

Next, we show that this argument can be iterated  to show $ u |_{\Sigma}$ is smooth 
in the case that $g$ is smooth up to $\S$ from both sides. By integrating again a test function, 
\eqref{eq-Delta-h-u} implies
\be
\Delta \, \p_\alpha u = - \p_\alpha g^{ij} \, \p^2_{ij} u 
+ \left( \p_\alpha g^{ij} \, \Gamma^k_{ij} + g^{ij} \p_\alpha \Gamma^k_{ij} \right) \p_k u .
\ee
Therefore,
\be \label{eq-Delta-h-u-1}
\begin{split}
\Delta \, \delta^h_\beta \p_\alpha u = & \  - \delta^h_\beta \p_\alpha g^{ij} \, \p^2_{ij} u - \tau^h_\alpha \p_\alpha g^{ij} \, \delta^h_\beta \p^2_{ij} u \\
& \ + \delta^h_\beta \left( \p_\alpha g^{ij} \, \Gamma^k_{ij} + g^{ij} \p_\alpha \Gamma^k_{ij} \right) \p_k u \\
& \  +  \tau^h_\beta \left( \p_\alpha g^{ij} \, \Gamma^k_{ij} + g^{ij} \p_\alpha \Gamma^k_{ij} \right) \, \delta^h_\beta \p_k u .
\end{split}
\ee
Here $ \tau^h_\beta $ is the operator of shifting the coordinate along $ e_\beta$ by $h$, i.e. $ \tau^h_\beta f (x) = f ( x + h e_\beta) $. 
Recall that, on any subdomain $ \tilde V \subset \subset V$, the preceding step has shown 
$ || \delta^h_\beta u ||_{W^{2,p} ( \tilde V) } \le C $, where $C$ is independent on $h$. This, combined with the fact $  \delta^h_\beta \p^2_{ij} u  = \p^2_{ij}  \delta^h_\beta  u  $  and the assumption that $g$ is smooth up to $\Sigma$ on both sides, then shows  the right side of \eqref{eq-Delta-h-u-1} is uniformly bounded in $L^p ({ \tilde V})$. Hence, by Theorem 9.11 in \cite{GT}, 
$ || \delta^h_\beta \p_\alpha u ||_{W^{2,p}  (\tilde V)}$ is uniformly bounded and consequently $ \p^2_{\beta \alpha} u \in W^{2,p}_{loc} (V)$.
By Sobolev embedding, $ \p^2_{\beta \alpha} u \in C^{1,\alpha}_{loc} (V)$, and this implies $ u |_{\Sigma} \in C^{3, \alpha} (\Sigma)$.
Repeating these arguments, we conclude $ u |_\Sigma \in C^\infty (\Sigma)$.

The remaining claim that $u$ is $C^{2,\alpha}$ ($C^\infty$) up to $\Sigma$ on the  two sides of $\Sigma$ follows directly from 
the standard  elliptic boundary regularity theory, see Theorem 6.19 in \cite{GT}.
\end{proof}

\begin{remark}
Proposition \ref{prop-u-corner} is a local result that holds for any $W^{2,p}$ harmonic function  in a neighborhood of the corner hypersurface $\Sigma$.
It continues to hold in higher dimensions as the proof does not require $\Sigma$ to be $2$-dimensional. 
{The  proof technique also works for more general linear, elliptic PDEs on manifolds with corner type singularity.}
\end{remark}

\begin{remark} \label{rem-C-2-u}
Let $H$ and $H_\Omega$ be the mean curvature of $\Sigma$ in $(E, g)$ and $(\Omega, g)$, respectively. 
Evaluating $\Delta  u = 0 $ along $\Sigma$ in its both sides  in $(M,g)$, one has
$$ 0 = \Delta_{_\Sigma} u + H \frac{\p u}{\p t} + \frac{\p^2 u}{\p t^2}  
\ \ \text{and} \ \
0= \Delta_{_\Sigma} u + H_{_\Omega} \frac{\p u}{\p t} + \frac{\p^2 u}{\p t^2},
$$
where $\Delta_{_\Sigma} u $ is the Laplacian of $ u |_{_\Sigma} $ on $\Sigma$ and $\p_t = \mu$ is the $\infty$-pointing unit normal
in $(M,g)$. This implies, if $H = H_{_\Omega}$, $u$ is $C^2$ across $\Sigma$ in the given coordinates $\{ x_\alpha, t \}$.
\end{remark}

\subsection{Mass of manifolds with corner surface} \label{P1}

We analyze the mass of $(M, g)$ following the method of Bray-Kazaras-Khrui-Stern \cite{BKKS}. For this purpose,
we impose the following topological assumption on $M = E \cup \Omega$:

\vspace{.1cm}

{\bf Condition (T)}: Assume $M$ contains no non-separating two-spheres.

\vspace{.1cm}

As $ u  \in W^{2,p}_{loc} (M) \subset C^1 (M)$, $ \Sigma_t = u^{-1} (t)$ is $C^1$ surface  for each regular value $t$.
Moreover, $u$ satisfies the maximum principle (see Theorem 8.1 in \cite{GT} for instance).
Thus, it follows from condition (T) and the maximum principle that
$\Sigma_t$ consists of a single non-compact connected component, together with 
some compact components (possibly empty) that are closed surfaces with positive genus.
Therefore, $ \chi (\Sigma_t) \le 1 $ (cf. Section 6 in \cite{BKKS}).

By Proposition \ref{prop-u-corner}, we can apply Proposition \ref{prop-mass-bdry} and Proposition  \ref{prop-bdry-1}  
to  $ (E, g)$ and $(\Omega, g)$, respectively, to obtain
\be \label{eq-mass-bdry-E}
\begin{split}
& \ 8 \pi \m +  2 \pi  \int_{-T}^T  ( \chi(\Sigma_{t, E}) - 1 )  \, dt \\
\ge & \ \int_{E}  \frac{1}{2} \left[ \frac{ | \nabla^2 u |^2 }{  | \nabla u |  }  +  R | \nabla u |   \right] 
 - \int_{\Sigma} H_{_E} | \nabla u |  \\
 & \ +  \int_{ \{  \nabla_{_{\Sigma} } u \ne 0 \} } \la \nabla_{_{\Sigma}} \p_\mu u, \frac{ \nabla_{_{\Sigma} } u  }{ | \nabla u |}  \ra  
  - \frac{ \p_\mu u }{ | \nabla u | }    \nabla^2_{_{ \Sigma } } u  ( \frac{ \nabla_{_{ \Sigma} } u  }{ | \nabla_{_{\Sigma}} u | }   , \frac{ \nabla_{_{\Sigma} } u  }{ | \nabla_{_{\Sigma}} u | }  ) 
\end{split}
\ee
and
\be \label{eq-prop-bdry-omega}
\begin{split}
& \ 2 \pi \int_{\min_{_\Omega} u }^{ \max_{_\Omega} u  } \chi(\Sigma_{t, \Omega})  \, dt  \\
 & \ +  \int_{ \{  \nabla_{_{\Sigma} } u \ne 0 \} } \la \nabla_{_{\Sigma}} \p_\mu u, \frac{ \nabla_{_{\Sigma} } u  }{ | \nabla u |}  \ra  
  - \frac{ \p_\mu u }{ | \nabla u | }    \nabla^2_{_{ \Sigma} } u ( \frac{ \nabla_{_{ \Sigma} } u }{ | \nabla_{_{\Sigma}} u | }   , \frac{ \nabla_{_{ \Sigma} } u }{ | \nabla_{_{\Sigma}} u | }  ) \\
\ge & \ \int_{\Omega}  \frac{1}{2} \left[ \frac{ | \nabla^2 u |^2 }{  | \nabla u |  }  +  R | \nabla u |   \right] 
+ \int_{\Sigma} H_{_{\Omega}}  | \nabla u |  .
\end{split}
\ee
Here $t$ is chosen to be a regular value of $ u $ and $u|_{_\Sigma} $, $ \Sigma_{t, E} = \Sigma_t \cap E $,
$ \Sigma_{t, \Omega} = \Sigma_t \cap \Omega$. 
Since $ u $ is $C^1$ on $M$ and $ C^2$ on $ \Sigma$, 
adding \eqref{eq-mass-bdry-E} and \eqref{eq-prop-bdry-omega} gives
\be \label{eq-mass-M-1}
\begin{split}
& \ 8 \pi \m +  2 \pi   \int_{- T}^T  \left[ \chi(\Sigma_{t, E})  + \chi(\Sigma_{t, \Omega} ) -1 \right]  \, dt \\
\ge & \ \int_{E}  \frac{1}{2} \left[ \frac{ | \nabla^2 u |^2 }{  | \nabla u |  }  +  R | \nabla u |   \right] 
 + \int_{\Omega}  \frac{1}{2} \left[ \frac{ | \nabla^2 u |^2 }{  | \nabla u |  }  +  R | \nabla u |   \right] 
 + \int_{\Sigma} ( H_{_\Omega} - H_{_E} ) | \nabla u |  .
\end{split}
\ee
Note that, for each chosen $t$, $ \Sigma_{t, E} \cup \Sigma_{t, \Omega} = \Sigma_t $
 if $ \Sigma_t \cap \Omega \ne \emptyset$; and 
  $ \Sigma_{t, E} = \Sigma_t$ if $ \Sigma_t \cap \Omega = \emptyset $. In either case, 
 $ \chi (\Sigma_{t, E} ) + \chi (\Sigma_{t, \Omega} ) = \chi (\Sigma_t) \le 1 $.
  Therefore, it follows from \eqref{eq-mass-M-1} that
\be \label{eq-mass-M-2}
8 \pi \m \ge  \int_{E}  \frac{1}{2} \left[ \frac{ | \nabla^2 u |^2 }{  | \nabla u |  }  +  R | \nabla u |   \right] 
 + \int_{\Omega}  \frac{1}{2} \left[ \frac{ | \nabla^2 u |^2 }{  | \nabla u |  }  +  R | \nabla u |   \right] 
 + \int_{\Sigma} ( H_{_\Omega} - H_{_E} ) | \nabla u |  ,
\ee
and equality holds only if $ \chi(\Sigma_t ) = 1 $ for a.e. $t$.
This proves Theorem \ref{main-1-intro}.

\begin{remark} \label{rem-mass-equation}
If the gradient $ | \nabla u |$ never vanishes on $M$, it is clear from its proof that equality holds in  \eqref{eq-mass-M-2}.
\end{remark}

\begin{remark}
The proof allows $E$ and $\Omega$ to have additional boundary components if those components have zero mean curvature. In this case, one can impose a Neumann boundary condition at the minimal boundary portion (see \cite{BKKS}) or impose a suitable Dirichlet boundary condition (see \cite{HKK}).
\end{remark}

\subsection{Positivity of the mass}
We give a proof of Corollary \ref{cor-WY} using \eqref{eq-main-intro}.
We keep the same notations as in the previous subsection, that is  $E$ denotes the asymptotically flat piece, 
and metrics on both  $E$ and $\Omega$ are denoted by $g$. 

Suppose there are vector fields $X$, $Y$ on $\Omega$, $E$, respectively, satisfying 
\be \label{eq-R-Ci-2-3}
R \ge \frac{2}{3} | X|^2 - 2 \, \div X \ \  \text{and}  \ \ R \ge \frac{2}{3} |Y|^2 - 2\, \div Y,
\ee
where $Y=O(|x|^{-1-2\tau})$ for some $\tau>\frac{1}{2}$. Here we consider the case $C_1 = C_2 = \frac{2}{3}$.
\vh

Let $(M, g)$, $ u$  be the glued manifold, the harmonic function as in the preceding proof of Theorem \ref{main-1-intro}.
On $(\Omega, g)$, $ | \nabla u |$ is Lipschitz, integrating by parts gives 
\be\label{EQ}
\begin{split}
& \  \int_{\Omega} \frac12  \left( \frac{1}{| \nabla u |} | \nabla^2 u |^2 + R | \nabla u | \right)  \\
\ge & \ \int_{\Omega} \frac{1}{ 2 | \nabla u |} | \nabla^2 u |^2 + \frac{1}{3}| X|^2 | \nabla u | + \la X, \nabla | \nabla u | \ra 
 - \int_{\S}  \la X , \nu  \ra  | \nabla u |.
\end{split}
\ee

To estimate the integral on $\Omega$, we use a refined Kato inequality
for harmonic functions (c.f. Lemma 3.1 in \cite{HKK} for instance).

\begin{lemma} \label{lem-K-ineq}
If $u$ is a harmonic function on a Riemannian manifold $(U^n , g)$, then
\begin{equation}
|\nabla^2u|^2\geq \frac{n}{n-1} |\nabla|\nabla u||^2 \ \ \mathrm{on} \ U \setminus D .
\end{equation}
Here $ D \subset \{ \nabla u = 0 \}$ is a set of measure zero at which $ | \nabla u |$ is not differentiable. 
If equality holds at some $ p \in U \setminus D$, then 
\begin{itemize}
\item[i)] $ \nabla^2 u (p) = 0 $ if $ \nabla u (p) = 0 $; and
\item[ii)] $ \displaystyle   \nabla^2 u (p) = \frac{n}{n-1}\nabla^2 u(\nu, \nu) (p) 
\lf( \nu \otimes \nu - \frac{1}{n} g \ri) (p)  $, if $ \nabla u (p) \ne 0 $. Here $ \displaystyle \nu = \frac{ \nabla u (p) }{ | \nabla u (p) |}$ and  $ \nu \otimes \nu $ is understood by 
$ \nu \otimes \nu (v , w) = \la \nu, v \ra \la \nu, w \ra $ for any $ v, w \in T_p U$.
\end{itemize}
\end{lemma}

\begin{proof}
Let $u_{;i} =\partial_i u$ and $ Z =\frac{1}{2} \nabla  |\nabla u|^2$, then
\be \label{eq-Z-norm}
\begin{split}
| Z |^2 = & \ \frac12 ( u_{;j} Z_i + u_{;i} Z_j )  (\nabla^2 u)^{ij} \\
= & \ \left[ \frac12 ( u_{;j} Z_i + u_{;i} Z_j ) - \frac{1}{n} \nabla^2 u (\nabla u , \nabla u) g_{ij}  \right]  (\nabla^2 u)^{ij},
\end{split}
\ee
where $ \Delta u = 0 $ was used in the second step.
Define
\be
W_{ij} = \frac12 ( u_{;j} Z_i + u_{;i} Z_j ) - \frac{1}{n} \nabla^2 u (\nabla u , \nabla u) g_{ij} , 
\ee
then $ g^{ij} W_{ij} = 0 $ and 
\be \label{eq-W-norm}
\begin{split}
|W|^2 = & \ \frac{1}{2}|Z|^2 |\nabla u|^2 + \left( \frac12 - \frac1n \right) \langle Z,\nabla u\rangle^2 \\
\leq & \ \frac{n-1}{n}  |Z|^2 |\nabla u|^2.
\end{split}
\ee
By \eqref{eq-Z-norm} and \eqref{eq-W-norm},
\begin{align} \label{eq-ZW}
\begin{split}
|Z|^2 = & \ \la W, \nabla^2 u \ra 
\leq  \sqrt{\frac{ n -1 }{ n }} |Z| |\nabla u| |\nabla^2u|,
\end{split}
\end{align}
which gives
$  |Z|\leq\sqrt{\frac{n-1}{n}}|\nabla u||\nabla^2u| $.
This readily shows
\begin{equation} \label{eq-Kato-ineq}
|\nabla|\nabla u||^2\le \frac{n-1} {n} |\nabla^2u|^2 , \ \ \mathrm{on} \ U \setminus D. 
\end{equation}
 
Suppose equality in \eqref{eq-Kato-ineq} holds at some $ p \in U \setminus D$.
If $ \nabla u (p) = 0 $, then 
$ \nabla | \nabla u | (p) = 0 $ which shows $ \nabla^2 u (p) = 0 $. 

Next suppose $ \nabla u (p) \ne 0 $. 
From the proof of \eqref{eq-Kato-ineq}, we have 
\be \label{eq-kato-1}
\la Z,  \nabla u \ra^2 (p) = | Z (p) |^2 \, | \nabla u(p) |^2 \ \  \text{and} \ \
 \la W, \nabla^2 u \ra (p) = | W (p) | \, | \nabla^2 u (p) | .
\ee
If $ \nabla^2 u (p) = 0$,  ii) holds automatically. 
Suppose $ \nabla^2 u(p) \neq 0 $. In this case, 
$ Z(p) \ne 0$. For otherwise $ \nabla | \nabla u | (p) = 0 $ which would imply
$\nabla^2 u (p) = 0 $ by the equality in \eqref{eq-Kato-ineq}.
Thus, \eqref{eq-kato-1} is equivalent to
\be \label{eq-kato-2}
Z(p) = \phi (p) \, \nabla u(p) \ \ \text{and} \ \ W (p) = \psi (p) \nabla^2 u (p),
\ee 
where $ \phi (p) $ and $ \psi (p)$ are some numbers with $ \phi (p) \neq 0$ and $ \psi (p) \ge 0$.

By the definition of $ W$ and the fact $ g^{ij} W_{ij} = 0 $, \eqref{eq-kato-2} gives
$ \phi (p) = \nabla^2 u ( \nu, \nu ) (p) $
and
\be \label{eq-kato-3}
 \phi (p) \lf( d u \otimes d u - \frac{1}{n} | \nabla u |^2 g \ri) (p) = \psi(p) \nabla^2 u (p).
\ee
Evaluating \eqref{eq-kato-3} along $(\nu, \nu)$ shows
$ \psi (p) =  \frac{n-1}{n}  | \nabla u (p)|^2 $. Part ii) follows.
\end{proof}

Coming back to \eqref{EQ}, by  Lemma \ref{lem-K-ineq}, we have
\be\label{EQ-2}
\begin{split}
& \ \int_{\Omega} \frac{1}{ 2 | \nabla u |} | \nabla^2 u |^2 + \frac{1}{3}| X|^2 | \nabla u | + \la X, \nabla | \nabla u | \ra \\
\ge & \ \int_{\Omega} \lf( \frac{\sqrt{3}}{ 2 \sqrt{ | \nabla u | } } | \nabla | \nabla  u  | | - 
\frac{1}{\sqrt{ 3} } |X | \sqrt{ | \nabla u |} \ri)^2 + | \nabla | \nabla u | | \,  |X|  + \la X, \nabla | \nabla u | \ra \\
\ge & \ 0 . 
\end{split}
\ee
Similarly, on $ E$, 
\be\label{EQ-3}
\begin{split}
& \ \int_{E} \frac{1}{ 2 | \nabla u |} | \nabla^2 u |^2 + \frac{1}{3}| Y|^2 | \nabla u | + \la Y, \nabla | \nabla u | \ra \\
\ge & \ \int_{E} \lf( \frac{\sqrt{3}}{ 2 \sqrt{ | \nabla u | } } | \nabla | \nabla  u  | | - 
\frac{1}{\sqrt{ 3} } |Y | \sqrt{ | \nabla u |} \ri)^2 + | \nabla | \nabla u | | \,  |Y|  + \la Y, \nabla | \nabla u | \ra \\
\ge & \ 0 . 
\end{split}
\ee
Thus, by \eqref{eq-mass-M-2}, \eqref{EQ}, and assumption c), 
\be\label{EQ4}
8 \pi\m \geq  \ \int_{\S} \left( H_{_\Omega} -\la X, \nu  \ra -H + \la Y , \mu \ra  \right) | \nabla u | \ge 0 .
\ee

\subsection{The zero mass case}
Suppose  $\m=0$ in the above context.
 Tracing all inequalities, we have
\be \label{eq-Euler-rigidity}
\chi (\Sigma_t) = 1, \ a.e. \ t  \in \R,
\ee
\be \label{eq-R-X-1}
\lf( R - \frac23  |X|^2 + 2 \div X \ri) | \nabla u | = 0 , \ \text{on} \ \Omega,
\ee
\be \label{eq-Kato-rigidity}
\frac{1}{ | \nabla u |} \lf( |\nabla^2u|^2 -  \frac{3}{2} |\nabla|\nabla u||^2 \ri) = 0 , \ \ a.e. \ \text{on} \ \Omega ,
\ee
\be \label{eq-nu-X-1}
\frac{\sqrt{3}}{ 2 \sqrt{ | \nabla u | } } | \nabla | \nabla  u  | | =
\frac{1}{\sqrt{ 3} } |X | \sqrt{ | \nabla u |}, \ a.e. \ \text{on} \ \Omega
\ee
and
\be \label{eq-nablau-X}
 | \nabla | \nabla u | | \,  |X|  + \la X, \nabla | \nabla u | \ra =0, \ a.e. \ \text{on} \ \Omega.
\ee
Analogous relations hold for $ R$, $u$ and $Y$ on $E$.

To show the rigidity case in Corollary \ref{cor-WY}, we assume 
$$
 R \ge C_1 |X|^2 - 2 \div X \  \text{on} \ \Omega
$$
for some constant $C_1 > \frac23$.  
Under this assumption, \eqref{eq-R-X-1} implies $|X| | \nabla u | = 0$ on $\Omega$. 
By \eqref{eq-nu-X-1}, $ \nabla | \nabla u | = 0 $
 a. e. on $\Omega$, and  hence $|\nabla u |$ is a constant on $\Omega$. 
Similarly, if 
$$ R \ge C_2  |Y|^2 - 2 \div Y \   \text{on} \ E$$ 
for some $C_2 > \frac23$, then $ | \nabla u |$ 
is a constant on $ E$.
Consequently, $|\nabla u | $ equals the same constant on $E$ and $\Omega$ since $| \nabla u |$ is $C^0$ on $M$.
As $ | \nabla u | \to 1$ near infinity, we have $ | \nabla u |= 1 $ on $M$. As a result, $ X =0 $, $Y=0$,
  $ R \ge 0 $ and  $H_{_\Omega} \ge H$. 
Theorem \ref{main-1-intro} and $\m = 0 $ then imply $ \nabla^2 u = 0 $, $ R=0$  and 
$H_{_\Omega} = H$. 

Let $ u_i$ denote the harmonic function $u$ with $x_1$ replaced by $ x_i$, $i=1,2,3$. 
The preceding proof shows that  $\{ \nabla u_1, \nabla u_2, \nabla u_3\}$, being linearly independent near infinity, are parallel on $(E, g)$ and $(\Omega, g)$.
Hence, $(E, g)$ and $(\Omega, g)$ are flat. 

We claim that $\Sigma$ has the same second fundamental form in $(E, g)$ and $(\Omega, g)$.
To see this, on each side of $\Sigma$ in $(M, g)$,  we have
\be \label{eq-same-Hessian}
0 =  \nabla^2 u_i (\p_\alpha , \p_\beta) = \nabla^2_{_\Sigma} u_i (\p_\alpha , \p_\beta) + \Pi (\p_\alpha , \p_\beta) \frac{\p u_i}{\p \mu},
\ee
where $\nabla^2_{_\Sigma}$ is the Hessian on $\Sigma$ and $\Pi$ is the corresponding second fundamental form of $\Sigma$. 
As $\{ \nabla u_1, \nabla u_2, \nabla u_3\}$ form a basis for $T_p M $ at each $ p \in \Sigma$, one of $u_i$ satisfies 
\be \label{eq-nonzero-mu-u}
 \frac{\p u_i}{\p \mu} (p) \neq 0 . 
\ee
It follows from \eqref{eq-same-Hessian} and \eqref{eq-nonzero-mu-u} that $\Pi$ are the same from both sides of $\Sigma$ in $(M, g)$.
This combined with the fact $g$ is flat near $\Sigma$ implies that  $g$ is $C^2$ across $\Sigma$ (see \cite{ST} for instance). 
On the other hand, $H_{_\Omega } =  H$ implies $u $ is $C^2$ across $\Sigma$ by Remark \ref{rem-C-2-u}. 
The smoothness of $u$ is further promoted  to be $C^{2,\alpha}$ via the regularity of $g$.
We  then proceed as  in \cite{BKKS} to conclude that the map 
$ F=(u_1, u_2, u_3)$ is a $C^{2, \alpha}$ isometry between $(M, g)$ and $\R^3$. 

\begin{remark}
In the setting of Theorem \ref{main-1-intro}, Corollary \ref{cor-WY} establishes the positive mass theorem with corner along $\Sigma$
(c.f. \cite{M1, ST, MS}). However, it does not seem to be a replacement of the later results due to the dimension restriction and the topological condition (T).
\end{remark}

The preceding proof does not make use of 
\eqref{eq-Euler-rigidity}, \eqref{eq-Kato-rigidity} and \eqref{eq-nablau-X}.
Exploring these relations, we can prove a rigidiity result 
for the case $ C_1 = C_2 = \frac23$, but under a stronger regularity assumption of $g$ across $\Sigma$.

\begin{theorem} \label{thm-rigidity-2-3}
Let $(M, g)$ be a smooth asymptotically flat $3$-manifold, without boundary,
which contains no non-separating $2$-spheres. 
Let $ \Omega \subset M$ be a bounded domain with smooth boundary $ \Sigma$.
Suppose there exist vector fields $X$, $Y$ on $ \bar \Omega$, $M \setminus \Omega$, respectively, such that 
\begin{itemize}
\item[a)]  $ R \ge \frac23 |X|^2-2\div X \ \text{on} \ \Omega $,
\item[b)] $  R \ge \frac23 |Y|^2-2\div Y \ \text{on} \  M \setminus \Omega 
$ with $Y=O(|x|^{-1-2\tau})$, and
\item[c)] $ \la Y,\mu  \ra \ge \la X,\nu \ra$, where $\nu$ is the $\infty$-pointing unit normal to $ \Sigma$.
\end{itemize}
Then $\m = 0 $ implies $(M , g)$ is isometric to $\R^3$.
\end{theorem}

\begin{proof} 
As we assume $g$ is smooth, the function $u$ is smooth. Hence each regular level set $\Sigma_t = u^{-1} (t)$ 
is a surface with a smooth metric induced from $g$.

Suppose $\m = 0 $. It suffices to show $ \nabla^2 u = 0 $ on $M$.
We first note
\be \label{eq-hu-gdu}
\nabla^2 u = 0 , \ a.e. \ \text{on} \ \{ x \in M \ | \ \nabla u(x) = 0 \} . 
\ee
This is because \eqref{eq-main-intro} implies 
$ 
\int_{M} \frac{ | \nabla^2 u |}{ | \nabla u |} < \infty $,
as $ R$ is integrable and $ | \nabla u | $ is bounded. 

Next, consider a set
$$A = \{ t \in \R \ | \ t \ \text{is a regular value of  both} \ u \ \text{and} \ u |_{\Sigma}, \ \text{and} \ \chi(\Sigma_t) = 1 \} .$$  
By \eqref{eq-Euler-rigidity}, the complement of $A$ in $ \R$ is a set of measure zero. 
The following is the main assertion in this proof:
\be \label{eq-uA-vH}
  \nabla^2 u (x) = 0, \ \forall \ x \in \Sigma_t \ \text{where} \ t \in A.  
\ee
If \eqref{eq-uA-vH} holds, then $ \nabla^2 u = 0 $ on $M$ follows. 
To see this, consider an open set
$$ B = \{ x \in M \ | \ \nabla^2 u (x) \ne 0 \} \cap \{ x \in M  \ | \ \nabla u (x) \ne 0 \} . $$
Let $ p \in B$ if $ B \ne \emptyset$. Consider a curve
$ \tau : (-\epsilon, \epsilon) \rightarrow B $,
for some small $ \epsilon > 0 $, with $ \tau (0) = p $ and $\tau'(0) = \nabla u (p)$.
As $\tau(s) \in B$ for $ s \in (-\epsilon, \epsilon)$, \eqref{eq-uA-vH} implies $ u ( \tau (s)  ) \notin A $. 
As a result, $\{ u (\tau (s) ) \ | \ s \in (-\epsilon, \epsilon) \}$ is a subset of measure zero in $ \R$. 
However, this set contains an open neighborhood of $ u(p)$ because
$  \frac{d}{d s} |_{s=0} u (\tau(s)) > 0 $. This is a contradiction. Therefore $B = \emptyset$.
Considering \eqref{eq-hu-gdu}, we then conclude $ \nabla^2 u = 0 $ on $M$.

The rest of the proof is to show \eqref{eq-uA-vH}.
We start with some local information near a connected surface $S$ on which $u$ is a constant with 
$ \nabla u \ne 0$.
By \eqref{eq-Kato-rigidity} and Lemma \ref{lem-K-ineq}, along $S$, $u$ satisfies
\be \label{eq-using-Kato}
\nabla^2 u = \frac{3}{2} \, \nabla^2 u(\nu, \nu) 
\lf( \nu \otimes \nu - \frac13 g \ri), \ \text{where} \ \nu = \frac{ \nabla u  }{ | \nabla u  |} .
\ee 
This implies
\begin{itemize}
\item $ |\nabla u |$ is a constant on $ S$; and 
\item $ | \nabla u | \, \Pi = - \frac12 \nabla^2 u(\nu, \nu) \gamma $, 
where $ \Pi$ is the second fundamental form of $ S $ with respect to $\nu$ and 
$\gamma$ is the metric on $S$ induced from $g$.
As a result,  the mean curvature $H $ of $S$ satisfies 
$ H | \nabla u | = - \nabla^2 u (\nu, \nu) $.
\end{itemize}
To get more information, we can consider the level set foliation of $u$ near $S$.
More precisely, let $ U $ be a bounded, connected open set in $S$ and let $W$ be an open neighborhood 
of $ U $ in $M$ obtained by flowing $ U$ along the integral flow of $ | \nabla u |^{-2}  \nabla u$.
$ W$ is diffeomorphic to $ (- \delta, \delta ) \times U $ for some $ \delta > 0$ and 
the metric $g$ on $ W$ can be written as
\be \label{eq-metric-gauge}
g = \frac{1}{ f ( s)^2 } d s^2 + \gamma_s .
\ee
Here $ f = | \nabla u | $ and $ \gamma_s $ is the induced metric on $ \{ s \} \times U$, 
which is part of the level set $\Sigma_{ t + s}$.
For convenience, we may assume $ u = 0 $ on $S$. Then $ u = s $ on $ W$, and 
\be \label{eq-nabla2-u-nu}
\nabla^2 u (\nu, \nu) = f^2 \, \nabla^2 s (\p_s, \p_s) = f f_{, s} ,
\ee
where $ f_{,s} = \frac{d f}{ds}$. Thus, $ \Pi$ and $ H$ satisfy 
\be \label{eq-Pi-and-H}
\Pi = - \frac12 f_{,s} \gamma \ \ \text{and} \ 
H = - f_{,s}  . 
\ee
On the other hand, \eqref{eq-metric-gauge} implies 
$ \Pi = \frac12 f \frac{d}{d s} \gamma_s $. 
Therefore, $ f \frac{d}{d s} \gamma_s + f_{,s} \gamma =  0 $ which shows 
\be
 f \gamma_s = \sigma  ,
\ee
where $\sigma$ is a fixed metric  on $ U$. As a result, on $ W = (-\delta, \delta) \times U $, 
\be \label{eq-R-scalar}
g = \frac{1}{ f(s)^2} d s^2 + \frac{1}{f(s) } \sigma . 
\ee

Next, we incorporate the scalar curvature $R$ and the vector fields $X$ and $Y$.
Suppose $ U \subset \subset S \cap \Omega$. Then, for $ \delta $ small,  $ W \subset \Omega$.
By \eqref{eq-R-X-1}, \eqref{eq-nu-X-1} and \eqref{eq-nablau-X}, on $ W$, we have
\be \label{eq-R-more}
 R =  \frac23  |X|^2 + 2 \div X  \ \ \text{and} \ \
 X = -  \frac{3}{2 | \nabla u | }  \nabla | \nabla u | .
\ee
This shows,  in terms of $ f = | \nabla u |$,  
\be
R = - \frac32 f_{,s}^2 + 3 f f_{,ss} . 
\ee
On the other hand, \eqref{eq-R-scalar} and direct calculation gives
\be
R = - \frac32 f_{,s}^2 + 2 f f_{,ss} + 2 f K ,
\ee
where $ K$ is the Gauss curvature of the metric $\sigma$ on $ U$. 
As a result,
\be \label{eq-fss}
f_{,ss} = 2 K . 
\ee
Consequently, $ \sigma$ is a metric of constant Gauss curvature on $ U$ and $ f_{,ss}$ equals a constant that 
is independent on $s$.
By using the vector field $Y$, we know the same conclusion holds on 
$ U \subset \subset S \cap ( M \setminus \bar \Omega )$.
As a result, if $ S$ intersects $\Sigma = \p \Omega$ transversely, then the induced metric on $S$ has constant Gauss curvature everywhere. 

We proceed to get information on $u$ near infinity.  
Since $ u $ is asymptotic to a coordinate function, we know there exists a compact set $T$, such that
\be \label{eq-gdu-T}
| \nabla u | = 1 \ \text{on} \ M \setminus T . 
\ee
This can be seen, for instance, by taking $u = u_1$ in $\{ u_i \}_{ 1 \le i \le 3 } $ where $(u_1, u_2, u_3)$ forms 
a harmonic coordinate chart near $\infty$. Then, outside a large coordinate ball $B_r$, each level set of $u$ is 
 $\{ u_1 = c \} \setminus B_r$ for some constant $c$. 
As $ | \nabla u | \to 1 $ at $\infty$ and $ | \nabla u | $ is locally a constant on each regular level set of $u$, we 
have $ | \nabla u | = 1 $ outside $B_r$.

It follows from \eqref{eq-gdu-T},  \eqref{eq-nabla2-u-nu} and \eqref{eq-Pi-and-H} that
\be \label{eq-nabla2-u-H}
\nabla^2 u  (\nu, \nu) = 0  \ \text{and} \ H = 0 \ \ \text{on} \ \Sigma_t \cap ( M \setminus T) 
\ee
for any regular level set $\Sigma_t$.
This, combined with \eqref{eq-using-Kato}, shows
\be
\nabla^2 u = 0 \  \text{on} \ M \setminus T . 
\ee

We are now in a position to prove \eqref{eq-uA-vH}.
Let $ S$ be a connected component of $ \Sigma_t$ with $t \in A$.
By \eqref{eq-Pi-and-H},  $ H$ is locally a constant and hence a constant on $S$. 
If $ S$ is non-compact,  then $ S \cap ( M \setminus T) \ne \emptyset $. 
By \eqref{eq-nabla2-u-H}, $H = 0 $ on $ S \cap ( M \setminus T) $.
Hence $H=0$ everywhere on $S$. 
By \eqref{eq-nabla2-u-nu} and \eqref{eq-Pi-and-H}, 
$ \nabla^2 u (\nu, \nu) = 0 $ on $S$. Therefore, by \eqref{eq-using-Kato}, 
$ \nabla^2 u = 0 $ on $S$.

If $ S$ is compact,  $S$ must be a torus.
This is because $ \chi(\Sigma_t) = 1$ and any higher genus surface has negative Euler characteristics.
Let $ \sigma $ be the induced metric on $S$. 
Since $ t $ is a regular value of both $ u$ and $ u|_\Sigma$, $ S$ intersects $\Sigma$ transversely.
Hence $ \sigma$ has constant Gauss curvature $K$.
By Gauss-Bonnet theorem, $ K = 0 $. 

Now consider an open neighborhood $W$ of $S$, diffeomorphic to $(-\delta, \delta) \times S$ 
for some $\delta >0$, on which $g$ takes the form of \eqref{eq-R-scalar}. 
By \eqref{eq-fss} and the fact $K=0$, we have
\be
\frac{d^2}{d s^2} f (s) = 0,  \ \forall \ s \in (-\delta, \delta). 
\ee
Therefore, $ f(s) = | \nabla u | (s) = a s + b $ for some constants $ a$ and $b$, with $ b = | \nabla u |$ on $ S$.

If $ a > 0$, the integral flow of $ | \nabla u |^{-2} \nabla u $ starting from $ S$ exits on $[0, \infty)$.
This means $u$ goes to $ + \infty$ along the flow. Hence the flow eventually enters  $M \setminus T$.
By \eqref{eq-gdu-T}, $ | \nabla u | (s) = 1 $ for large $s$. This contradicts with $ | \nabla u | (s) = a s + b$. 

If $ a < 0 $, by consider the integral flow of $ - | \nabla u |^{-2} \nabla u $ starting from $ S$, we again have 
a contradiction. 

Therefore, $ a = 0$, i.e. $f_{,s} = 0 $. This combined with \eqref{eq-nabla2-u-nu} shows 
$$ \nabla^2 u (\nu, \nu) = 0 \ \text{on} \ S . $$ 
We conclude $ \nabla^2 u = 0 $ on $ S$ by \eqref{eq-using-Kato}. 
This completes the proof of \eqref{eq-uA-vH}.
\end{proof}

\begin{remark}
If the metric is merely $C^{0,1}$ across $\Sigma$, like that in Corollary \ref{cor-WY}, 
new difficulties seem to arise in analyzing rigidity for the case  $C_1=C_2=\frac{2}{3}$.
For instance, the function $u$ is $ W^{2,p}$ (thus $ C^{1, \alpha}$) across $ \Sigma$. A regular level set $S$ of $u$, intersecting transversely with $ \Sigma$ at a curve $\mathcal{C}$, may not be $C^{2}$ across $\mathcal{C}$. 
In particular, the two Gauss curvature constants on both sides of $\mathcal{C}$ in $S$ might not agree. 

It is plausible to work with the condition $ | \nabla u | = 1 $, instead of $\nabla^2 u = 0 $, in the claim \eqref{eq-uA-vH}
since $ | \nabla u | $ is $C^0$ on $M$. By an argument similar to part of the proof of Theorem \ref{thm-rigidity-2-3},
it indeed can be shown that $ | \nabla u | = 1 $ along any non-compact component of $\Sigma_t$ for $ t \in A$. 
Thus, under a more stringent topological assumption that $M $
contains no non-separating spheres and torus,  the rigidity statement in Corollary \ref{cor-WY} 
holds for $C_1= C_2= \frac23$.
\end{remark}

\subsection{Fill-in of $(M, g)$ by its conformal deformation} \label{sec-conformal-fill-in}

Given an asymptotically flat $(E^3, g)$ with boundary $ \Sigma$, consider its Green function $G$, which is determined by
\be
\Delta_g G = 0 \ \text{on} \ E,  \  \ G = 1 \ \text{at} \ \Sigma , \ \ \text{and} \  G \to 0 \ \text{at} \ \infty .
\ee
Let $ g_* = G^4 g $. Let $ \Omega = E \cup \{q  \}$ denote the one-point compactification of $M$ by including a point $ q $ representing 
$\infty$. It follows from the  asymptotically flatness of $g$  and the asymptotic expansion of $G$  that $g_*$ is continuously extendable across $q$, and moreover, $g _* \in W^{1,p} (B)$ 
for some $ p > 6$ in a neighborhood $B$ of $ q$ (see Lemma 6.1 in \cite{MMT} for instance).

This $(\Omega, g_*)$ serves as a fill-in of $(E, g)$. The metric on $(E, g) \cup (\Omega, g_*)$, glued along $ \Sigma$, is a $ W^{1,p}_{loc} $ metric
on $E \cup \Omega$.
In particular,  the function $u$ in Section \ref{Existence} exists on $(E, g) \cup (\Omega, g_*)$. 
Moreover, at most one regular level set of $u$ passes through the point $q$. 
If $E$ is topologically $ \R^3 $ minus a ball,  $E \cup \Omega$ is homeomorphic to $ \R^3$. Therefore, 
the proof of Theorem \ref{main-1-intro} applies to show
\be \label{eq-E-Omega-*}
8 \pi \m \ge  \int_{E}  \frac{1}{2} \left[ \frac{ | \nabla^2 u |^2 }{  | \nabla u |  }  +  R | \nabla u |   \right] 
 + \int_{\Omega}  \frac{1}{2} \left[ \frac{ | \nabla^2 u |^2 }{  | \nabla u |  }  +  R | \nabla u |   \right] 
 + \int_{\Sigma} ( H_{*} - H ) | \nabla u |  .
\ee
Here  $H_*$ is  the mean curvature of $\Sigma$ in $(\Omega, g_*)$ with respect to the outward normal and is given by 
 \be
 H_* = -  H - 4 \frac{\p G}{\p \mu}  ,
 \ee
where $\mu $ is the $\infty$-pointing unit normal to $ \Sigma$ in $(M,g)$ and $H$ is the mean curvature of $ \Sigma$ in $(M, g)$ 
with respect to $ \mu$. As a result,
\be
H_* - H = - 2 H - 4 \frac{\p G}{\p \mu} . 
\ee

Now let $ u $ denote  the restriction of $u$ on $ E $ and let $ u_{_\Omega}$ be the restriction of $u$ to $ \Omega$.
On $ \Omega = M \cup \{ q \} $, let $ v = u_{_\Omega} G $, then
\be
 \Delta_g v = 0 ,  \ \   v \to 0 \ \text{at} \ \infty, \  \ v = u \  \text{at} \  \Sigma,
\ee
and
\be \label{eq-bu-1}
\frac{\p v}{\p \mu} =  \frac{\p u_{_\Omega} }{\p \mu} + u \frac{\p G}{\p \mu} .
\ee
As  $ u $ is $C^1$ on $(E, g) \cup (\Omega , g_*)$, 
\be \label{eq-bu-2}
\frac{\p u_{_\Omega}}{\p \mu} = - \frac{\p u}{\p \mu} .
\ee
(Here $\nu$ points inward in the conformal fill-in $(\Omega, g_*)$.)
Set $ w = u - v $, which satisfies 
\be
\Delta_g w = 0 \ \text{on} \ E, \ \ w = 0 \ \text{at} \ \Sigma, \ \ w \ \text{is asymptotic to } x_i \ \text{at} \ \infty. 
\ee
By \eqref{eq-bu-1}, \eqref{eq-bu-2} and the fact $ u = v $ at $ \Sigma$, we have
\be \label{eq-bdry-harmonic-functions}
2 \frac{\p v}{\p \mu} = v \frac{\p G}{\p \mu} - \frac{\p w}{\p \mu}. 
\ee
Corollary \ref{Green-intro} then follows from \eqref{eq-E-Omega-*} and \eqref{eq-bdry-harmonic-functions}.

\section{Convergence of the Brown-York mass} \label{sec-BY-mass}

In this section, we apply Theorem \ref{main-1-intro} to study convergence of the Brown-York mass of 
surfaces approaching infinity in an asymptotically flat $3$-manifold.

Recall that the Brown-York mass of a surface $\Sigma$ with positive Gauss curvature, with mean curvature $H$ in 
a $3$-manifold $(M, g)$, is 
\be
\m_{_{BY}} (\Sigma) = \frac{1}{8\pi} \int_{\Sigma} (H_0 - H) \, d \sigma, 
\ee
where $ H_0$ is the mean curvature of the isometric embedding of $\Sigma$ in the Euclidean space $ \R^3$.

The fact $ \m_{_{BY} } (\cdot)$ converges to the mass of an asymptotically flat manifold $(M, g)$ 
along large coordinate spheres was demonstrated by Fan, Shi and Tam \cite{FST}.
The convergence was later established for large nearly round surfaces by Shi, Wang and Wu \cite{SWW-09}.
Similar results for surfaces of revolution were obtained by Fan and Kwong \cite{FK-1, FK-2} in asymptotically 
Schwarzschild manifolds. 

Using Theorem \ref{main-1-intro}, we demonstrate the convergence of $\m_{_{BY} }(\cdot)$ for a large class of 
surfaces, which in particular includes surfaces obtained by scaling any fixed convex surface in the background Euclidean space.

\begin{theorem} \label{thm-mass-bymass}
Let $(M^3, g)$ be an asymptotically flat manifold with a coordinate chart $\{ x_i \}$ in which 
the metric coefficients $g_{ij}$ satisfies the asymptotically flat condition \eqref{eq-g-AF} with $k \ge 4$ and $ p > 3$.
Let $ \{  \Sigma_r \}_{ r \ge r_0}  $ be a $1$-parameter family of surfaces approaching the infinity with
$ r = \min_{\Sigma_r} \{ | x | \ | \ x \in \Sigma_r \} $.
Suppose the rescaled surfaces
$$
\widetilde \Sigma_r : = \left\{ \frac{1}{r} x \, | \, x \in \Sigma_r \right \} \subset \R^3
$$
satisfy curvature condition 
\be \label{eq-condition-kappa}
0 < k_1 < \kappa_\alpha < k_2, \ \alpha = 1, 2, 
\ee
where $ \{ \kappa_\alpha \}$ are the principal curvatures of $ \widetilde \Sigma_r $  with respect to the background Euclidean metric in $ \R^3$,
and $ k_1, k_2$ are two positive constants independent on $r$. 
Then 
\be \label{eq-BY-limit-1}
 \lim_{r \to \infty} \m_{_{BY} }( \Sigma_r) = \m ,
\ee
where $ \m $ is the mass of $(M, g)$. 
\end{theorem}

Note that condition \eqref{eq-condition-kappa} is equivalent to 
\be \label{eq-condition-kappa-same}
 \bar K > C_1 > 0 \ \ \text{and} \  0< \bar H < C_2 
\ee
for some constants $ C_1$, $ C_2$ independent on $r$. Here $\bar K$, $\bar H$ denotes the Gauss curvature, the mean curvature
of  $ \widetilde \Sigma_r $  with respect to the background Euclidean metric in $ \R^3$, respectively.

We first outline the idea behind the proof of Theorem \ref{thm-mass-bymass}. 
By condition \eqref{eq-condition-kappa} and the asymptotically flatness of $g$, $\Sigma_r$ 
has positive Gauss curvature for large $r$. Thus $\Sigma_r$ isometrically embeds in the Euclidean space $ (\R^3, g_0)$ as a convex surface 
$ \Sigma_r^{(0)}$ (\cite{Nirenberg, Pogorelov}). Let $\Omega_r^{(0)} $ be the domain 
enclosed by $\Sigma_r^{(0)}$ in $ \R^3$ and $E_r $ be the exterior region of $\Sigma_r$ in $M$. 
The Euclidean domain $(\Omega_r^{(0)}, g_0)$ provides a natural fill-in of $\Sigma_r = \p E_r$.
Consider the manifold $(N_r, g_r) $ obtained by gluing $(\Omega_r^{(0)}, g_0)$ and $(E_r, g)$ along $\Sigma_r $.
Intuitively, $ (N_r, g_r)$ approaches $(\R^3 , g_0)$ in a suitable sense as $r \to \infty$. Hence
applying Theorem \ref{main-1-intro} to $(N_r, g_r)$, we would expect the harmonic function $u$ on $(N_r, g_r)$ to 
approach a standard coordinate function on $(\R^3, g_0)$. In particular \eqref{eq-main-intro} would be an equation, 
 the bulk integral on its right side would tend to zero, and the surface integral would be $ \m_{_{BY} }( \Sigma_r)+ o (1)$. 
 
To implement this idea, we need to have uniform control on the isometric embedding and 
to impose a uniform gauge  when identifying $(E_r, g) \cup (\Omega_r^{(0)}, g_0)$ with $ (\R^3, g_r)$ so that 
a fixed background is available when estimating $u$. We use condition \eqref{eq-condition-kappa} for achieving both these purposes. 

\begin{proof}
For each large  $r$, let $M_{r}  = \{ | x | \ge  r \} \subset M$.
Consider the diffeomorphism
\bee
\Phi_r : M_{r}\longrightarrow E = \{ |x | \ge 1 \ | \ x \in \R^3 \}  ,
\eee
where $\Phi_r ( x) = r^{-1} x $.
Define a metric $ \tilde g_r$ on $E$ by
$ \tilde g_r = r^{-2} ( \Phi_r^{-1})^*(g) $, 
then 
\be
 {\tilde g}_{r \, ij} (x) = g_{ij} (rx) , \,  x \in E.
\ee

By definition, $ \widetilde \Sigma_r = \Phi_r ( \Sigma_r) \subset E$. 
Let $ g_0$ denote the background Euclidean metric on $ \R^3$. 
Let $\bar \sigma_r$ be the induced metric on $ \widetilde \Sigma_r$ in  $(\R^3, g_0)$ and 
$ \bar K $ be the Gauss curvature of $ ( \widetilde \Sigma_r, \bar \sigma_r)$.
By condition \eqref{eq-condition-kappa}, $ \bar K \ge k_1^2 > 0 $.
This shows the intrinsic diameter of $ (\widetilde \Sigma_r, \bar \sigma_r) $ is bounded above 
by a constant independent on $r$. Thus there exists a constant $ L \ge 1 $ so that $ \widetilde \Sigma_r$ 
lies between $\{ | x | = L \}$ and $ \{ | x | = 1 \}$ for all large $r$. 

In $E$, let  $ T $ be the closed annular region bounded by $ \{ | x | = L \}$ and $ \{ | x | = 1 \}$.
By the asymptotically flatness condition \eqref{eq-g-AF} and weighted Sobolev inequalities, 
\be \label{eq-tgr-g0-1}
|| \tilde g_r - g_0 ||_{C^{3,\alpha} (T)} = O (r^{-\tau} )
\ee
for some $ 0 < \alpha < 1$.
As a result, 
\be
|| \tilde \sigma_r -  \bar \sigma_r ||_{C^{3,\alpha} (\widetilde \Sigma_r)} = O ( r^{- \tau} ),
\ee
\be \label{eq-nu-r-1}
|| \tilde \nu - \bar \nu  ||_{C^3(\widetilde \Sigma_r)} = O (r^{-\tau} ) ,
\ee
and
\be \label{eq-Hest-H-0-1}
|| \tilde H - \bar H ||_{ C^0( \widetilde \Sigma_r)} = O (r^{-\tau} ) .
\ee
Here  $\tilde \sigma_r$, $ \tilde \nu$, $\tilde H$ denote the induced metric, the inward unit normal, 
the mean curvature of $ \widetilde \Sigma_r $ in $(T, \tilde g_r)$, and $ \bar  \nu$, $\bar H$ denote the inward unit normal, 
the mean curvature of $\widetilde \Sigma_r $ in $(T, g_0)$.

By the openness of solutions to the Weyl embedding problem \cite{Nirenberg, LW},
it follows from the conditions $ \widetilde \Sigma_r \subset T$ and $ k_1^2 < K < k_2^2 $
that, for large $r$, 
there is an isometric embedding 
$\widetilde X_r : (\widetilde \Sigma_r, \tilde \sigma_r) \rightarrow \R^3$
so that
\be \label{eq-tXr-1}
|| \widetilde X_r - \mathrm{Id}  ||_{C^{3, \alpha} (\widetilde \Sigma_r)  } \le C || \tilde \sigma_r - \bar \sigma_r ||_{C^{3, \alpha} 
(\widetilde \Sigma_r) } .
\ee
Here $ \mathrm{Id}: \R^3 \rightarrow \R^3 $ denotes the identity map and $ C$ is a constant independent on $r$.
(This estimate was stated for $C^{2,\alpha}$ in \cite{Nirenberg}, but examining sections 5-9 in \cite{Nirenberg} 
shows higher order derives estimate hold as well.)
Moreover, $\tilde X_r$ has the same regularity as the metric $\tilde \sigma_r$ (see section 3 in \cite{Nirenberg}.)

In what follows, we let $  \widetilde  \Sigma_r^{(0)}  = \widetilde X_r ( \widetilde \Sigma_r) $ and let $\widetilde  \Omega^{(0)}_r$ be  the region enclosed by 
$ \widetilde \Sigma_r^{(0)} $ in $ \R^3$. Let $ \tilde \nu_0 $ be the unit inward normal to $ \widetilde \Sigma_r^{(0)}$ in $(\widetilde \Omega^{(0)}_r, g_0)$.
By \eqref{eq-tXr-1}, 
\be \label{eq-nu-0-r-1}
|| \tilde \nu_0 - \bar \nu ||_{C^{2, \alpha} (\widetilde \Sigma_r) } = O (r^{-\tau} ),
\ee
and
\be \label{eq-Hest-H-0}
|| \tilde H_0 - \bar H ||_{C^{0} (\Sigma)} = O (r^{-\tau} ) . 
\ee
Here $ \tilde H_0$ denote the mean curvature of $ \widetilde \Sigma_r^{(0)}$ in $(\R^3, g_0)$.

The Euclidean domains $(\widetilde \Omega_r^{(0)}, g_0)$ provides a natural fill-in of $(\widetilde E_r, \tilde g_r)$,
where $ \widetilde E_r $ denotes the exterior of $\widetilde \Sigma_r$ in $E$. 
To derive estimates on solutions to the corresponding Laplacian equation,
 we need a gauge to identify $ \widetilde E_r \cup \widetilde \Omega_r^{(0)}$ with $ \R^3$, which can be estimated via $r$.
For this purpose,  let $ \widetilde \Omega_r$ be the region enclosed by $ \widetilde \Sigma_r$ in $\R^3$.
If we could extend $\widetilde X_r :  \widetilde \Sigma_r \rightarrow  \widetilde \Sigma_r^{(0)} $ to a diffeomorphism
$ F_r: \widetilde \Omega_r \rightarrow \tilde \Omega_r^{(0)}$,  
such that $ (F_r)_* ( \tilde \nu)  = \tilde \nu_0 $ and $ F_r$ remains close to $\mathrm{Id}$, 
then $ F_r^* (g_0)$ on $\widetilde \Omega_r$ 
and $\tilde g_r$ on $\widetilde E_r$ would form a locally Lipschitz metric $\hat g_r$ on $ \R^3 $, 
which is close to $g_0$. We may then carry out analysis on this  fixed $\R^3$ equipped with $\{ \hat g_r \}$.
We do not know how to construct such an map  $F_r: \widetilde \Omega_r  \rightarrow \widetilde \Omega_r^{(0)}$. 
Instead, we produce an alternative  map $\tilde F_r: \widetilde \Omega_r  \rightarrow \R^3$ below, which suffices for the purpose 
of  analyzing $\m_{_{BY} } (\Sigma_r)$.

By condition \eqref{eq-condition-kappa}, 
the second fundamental forms of $\widetilde \Sigma_r$ are uniformly bounded in $(\R^3, g_0)$, and so are the second fundamental 
forms of $\widetilde \Sigma_r^{(0)}$ in $(\R^3, g_0)$ by \eqref{eq-tXr-1}.
It follows that $\widetilde \Sigma_r$, $\widetilde \Sigma_r^{(0)}$ have Gaussian tubular neighborhoods in $(\R^3, g_0)$ which have a fixed width 
independent on $r$.  For small $ s \ge 0$,  consider the map
\be
\Phi_r ( \omega  + s \tilde \nu) = \widetilde X_r (\omega) + s \tilde \nu_0 .
\ee
By \eqref{eq-tgr-g0-1}, there exists a small $\epsilon > 0 $, independent on $r$, such that
\begin{itemize} 
\item  $ U_{r , \epsilon} = \{ \omega + s \tilde \nu \ | \ \omega \in \widetilde \Sigma_r, \, 0 \le s < 2 \epsilon \}$ is an open neighborhood of 
$\widetilde \Sigma_r$ in $\widetilde \Omega_r $,  diffeomorphic to $ \widetilde \Sigma_r \times [0, 2 \epsilon)$; and 

\item $ \Phi_r$ is a diffeomorphism from $U_{r, \epsilon}$ to a Gaussian neighborhood of $\widetilde \Sigma_r^{(0)}$ in $(\widetilde \Omega_r^{(0)}, g_0)$.
\end{itemize} 
Let $ \eta$ be a fixed smooth function on $ [0, \infty)$, independent on $r$, satisfying   $ 0 \le \eta \le 1 $, 
$ \eta = 1 $ on $[0, \frac12 \epsilon]$,
and $\eta = 0 $ on $[\epsilon , \infty)$. Define  $  \tilde F_r : \widetilde \Omega_r \rightarrow \R^3 $, which interpolates $\Phi_r$ and $ \mathrm{Id}$, by
\begin{itemize}
\item $ \tilde F_r (x) = \mathrm{Id} (x)$, if $ x \notin U_{r, \epsilon}$; 
\item $ \tilde F_r (x) = \Phi_r (x ) \eta (s) + \mathrm{Id} (x) ( 1 - \eta (s) ) $, if $ x = \omega + s \tilde \nu_r \in U_{r, \epsilon}$.
\end{itemize}
It follows from \eqref{eq-nu-r-1},  \eqref{eq-tXr-1} and \eqref{eq-nu-0-r-1}
that
\be \label{eq-gauge-F-nu}
|| \tilde F_r - \mathrm{Id} ||_{C^{2, \alpha} (\widetilde \Omega_r) } \le C (  || \tilde X_r - \mathrm{Id} ||_{C^{2, \alpha}(\widetilde \Sigma_r)} 
+ || \tilde \nu - \tilde \nu_0 ||_{C^{2, \alpha} (\widetilde \Sigma_r) } ) = O ( r^{-\tau} ).
\ee
In particular, for large $r$, $\tilde F_r: \widetilde \Omega_r \rightarrow \R^3$ is an immersion.
Define $ \tilde  g_r = \tilde F_r^* (g_0)  $ on $\widetilde \Omega_r$.  By construction, $(\widetilde \Omega_r, \tilde   g_r)$ satisfies
\begin{itemize}
\item[a)] the induced metric from $\tilde g_r $ on $\widetilde \Sigma_r$ is  $\tilde \sigma_r$; 

\item[b)] $ \tilde \nu $ is the inward unit normal to $\widetilde \Sigma_r$ in $(\widetilde \Omega_r, \tilde  g_r)$, since $ ( \tilde F_r)_* ( \tilde \nu) = \tilde \nu_0 $;
\item[c)] $ || \tilde  g_r - g_0 ||_{C^{1,\alpha}(\widetilde \Omega_r) } = O (r^{-\tau} ) $.
\end{itemize}
Moreover, the mean curvature of $ \widetilde \Sigma_r$ in $(\widetilde \Omega_r, \tilde g_r)$ equals $\tilde H_0$, which is the mean curvature
of $\tilde \Sigma_r^{(0)}$ in $(\R^3, g_0)$.

We now write  $( \R^3, \hat g_r) = (\widetilde E_r, \tilde g_r) \cup ( \widetilde \Omega_r, \tilde g_r )$, where $ \hat g_r$ is a $C^{0,1}_{loc}$ metric on $ \R^3$.
Moreover, by properties a), b), c) above, 
\be \label{eq-p-g-0-1}
|| ( \hat g_{r} - g_0)_{ij} ||_{W^{2, p}_{ - \tau} (\widetilde E_r) } = r^{-\tau} || g_{ij} - \delta_{ij}  ||_{W^{2, p}_{ - \tau} (E_r) } 
= O (r^{-\tau}),
\ee
where $E_r$ is the exterior of $\Sigma_r $ in $(M,g)$, and 
\be \label{eq-p-g-B}
|| \hat g_r - g_0 ||_{C^{1} (\widetilde \Omega_r)} = O ( r^{-\tau} ).
\ee
Consider the operator 
\be
\Delta_{\hat g_r}: W^{2,p}_{1 - \tau} ( \R^3) \longrightarrow W^{0,p}_{- 1 - \tau} ( \R^3).
\ee
Here $ W^{k, p}_\delta (\R^3)$ is the weighted Sobolev space on $ \R^3$, consisting of functions 
$ v \in W^{k,p}_{loc} (\R^3)$ such that 
\be
|| v ||_{W^{k,p}_{\delta} ( \R^3 ) } =\sum_{|\beta|\leq k} 
||  \partial^\beta u \, ( 1 + |x|^2)^{ \frac12 ( |\beta| - \d -\frac{3}{p} ) } ||_{L^p (\R^3 )} < \infty. 
\ee
By  \eqref{eq-p-g-0-1} and \eqref{eq-p-g-B}, $\Delta_{\hat g_r} x_1  \in W^{0, p}_{ - 1 -  \tau } (\R^3)  $
and
\be \label{eq-delta-x1}
|| \Delta_{\hat g_r} x_1||_{W^{0, p}_{ - 1 -  \tau } ( \R^3) }    = O ( r^{-\tau} ). 
\ee
By the usual construction of harmonic coordinates (see Section 3 in \cite{Bartnik} for instance), 
there exists $ w_r \in W^{2, p}_{1 - \tau} ( \R^3)$ so that 
\be
\Delta_{\hat g_r}  w_r = \Delta_{\hat g_r} x_1 .
\ee
Let $ u_r = x_1 - w_r  $, then $ u_r$ satisfies  $\Delta_{\hat g_r}  u_r = 0 $ and
$  u_r $ is asymptotic to $ x_1$.

We estimate $ w_r $ by applying Corollary 1.16 of \cite{Bartnik}. 
Shrinking $\tau $ if needed, we may assume $ \tau \in (\frac12, 1)$. 
By \eqref{eq-p-g-0-1} and \eqref{eq-p-g-B},
\be
\lim_{r \to \infty} ||  \hat g_r - g_0 ||_{W^{1,p}_{0} (\R^3)} = 0 . 
\ee
By  Corollary 1.16 in \cite{Bartnik}, 
\be \label{eq-Bartnik}
|| w_r - \mathrm{Ker} ( {\Delta_{\hat g_r } } )   ||_{W^{2,p}_{1 -\tau} (\R^3) } \le C || \Delta_{\hat g_r}  w_r  ||_{W^{0,p}_{- 1 - \tau } (\R^3) } .
\ee
Here $ \mathrm{Ker} ( {\Delta_{\hat g_r } } )  $ is the Kernel of $\Delta_{ \hat g_r} : W^{2, p}_{1 - \tau} (\R^3) 
\rightarrow W^{0,p}_{-1 - \tau} (\R^3)  $,
and $C$ is some constant independent on $r$.
Since $ 1 - \tau \in (0 , \frac12)$, we have
\be
\dim \mathrm{Ker} ( {\Delta_{\hat g_r } }  ) = \dim \mathrm{Ker} ( {\Delta_{g_0 } } ) =1 
\ee
(see Proposition 1.15 and Corollary 1.9 of \cite{Bartnik}). 
Thus, 
\be
\mathrm{Ker} ( {\Delta_{\hat g_r } } )  = \{ \text{constant functions} \} .
\ee
This, combined with \eqref{eq-Bartnik}, gives
\be
\begin{split}
\inf_{c \in \R } || w_r - c ||_{ W^{2, p}_{1 - \tau} (\R^3)  } \le & \ C    || \Delta_{\hat g_r}  w_r  ||_{W^{0,p}_{- 1 -  \tau } (\R^3)  } \\
= & \ C  || \Delta_{\hat g_r}  x_1  ||_{W^{0,p}_{- 1 - \tau } (\R^3)  } .
\end{split}
\ee
We pick $c_r \in \R$ as follows:
if $  || \Delta_{\hat g_r}  x_1  ||_{W^{0,p}_{- 1 - \tau }  (\R^3) }  > 0 $, pick $c_r \in \R$ so that
\be
|| w_r - c_r ||_{ W^{2, p}_{1 - \tau} } \le 2 C  || \Delta_{\hat g_r}  x_1  ||_{W^{0,p}_{- 1 - \tau } (\R^3)  } .
\ee
If $  || \Delta_{\hat g_r}  x_1  ||_{W^{0,p}_{- 1 - \tau } (\R^3) }  = 0 $,
we simply pick  $ c_r = 0 $, in which case we also let  $ w_r = 0$.
For these choices of $w_r$ and $c_r$, we have
\be
|| w_r - c_r ||_{ W^{2, p}_{1 - \tau} (\R^3) } \le 2 C  || \Delta_{\hat g_r}  x_1  ||_{W^{0,p}_{- \tau -1} } .
\ee
Together with \eqref{eq-delta-x1}, this shows 
\be  \label{eq-pw-1}
|| \p  w_r   ||_{ W^{1, p}_{- \tau} (R^3) } = O ( r^{-\tau} ).
\ee

We now estimate the gradient  of $ u_r = x_1 - w_r $. 
By  \eqref{eq-pw-1} and Sobolev embedding, 
\be
 | \nabla  u_r   -  \nabla x_1 |_{\hat g_r}  = | \nabla   w_r  |_{\hat g_r}  =  O ( r^{-\tau} ).
\ee
As  $ | \nabla x_1 |_{\hat g_r} = 1 + O ( r^{-\tau} )$, this gives 
\be \label{eq-grad-ur}
| \nabla  u_r |_{\hat g_r} = 1 +  O ( r^{-\tau} ).
\ee
In particular, $ | \nabla  u_r |_{\hat g_r}  > 0 $ on $ \R^3$, for large $r$.
By Theorem \ref{main-1-intro} and Remark \ref{rem-mass-equation}, we have
\begin{align}
\label{eq-main-by}
\m (\hat g_r) =   \frac{1}{16 \pi} \int_{ \R^3   }\left[\frac{|\nabla^2 u_r |_{\hat g_r} }{|\nabla u_r |_{\hat g_r} } 
+R_{\hat g_r}  |\nabla u_r |_{\hat g_r} \right] \, d V_{\hat g_r} + \frac{1}{8\pi} \int_{\widetilde \Sigma_r}  ( \tilde H_{0} - \tilde H ) |\nabla u_r |_{\hat g_r}  \,
d \hat \sigma_r ,
\end{align}
where 
\be
 \m( \hat g_r) = r^{-1} \m (g) \ \ \text{and} \ \ \frac{1}{8\pi} \int_\S ( \tilde H_{0} - \tilde H )  \, 
 d \hat \sigma_r  = r^{-1}  \m_{_{BY} } (\Sigma_r)  .
\ee

We estimate the other terms in the right side of \eqref{eq-main-by}.
By \eqref{eq-grad-ur},
\be
 \int_{\widetilde \Sigma_r} ( \tilde H_0 - \tilde  H ) ( | \nabla  u_r |_{\hat g_r}  - 1 ) \, d \hat \sigma_r  = O ( r^{-2 \tau} ). 
\ee
Here we used  \eqref{eq-Hest-H-0-1} and \eqref{eq-Hest-H-0}, which together imply 
\be
| \tilde  H  - \tilde H_0 | =  O ( r^{ - \tau} ).
\ee

For the integral involving  $R_{\hat g_r}$, we have
\be
\int_{\widetilde \Omega_r} R_{\hat g_r} | \nabla  u_r |_{ \hat g_r} \, d V_{\hat g_r}  = 0 
\ee
because,  $ \hat g_r = \tilde F_r^* (g_0)  $  on $\widetilde \Omega_r$, which is flat. 
On $\widetilde E_r$, by \eqref{eq-grad-ur}, 
\be
\begin{split}
\int_{ \widetilde E_r}  R_{\hat g_r}  |\nabla u_r |_{\hat g_r}  \, d V_{\hat g_r}  \le & \ 
C  \int_{ \widetilde E_r }  R_{\hat g_r}  \, d V_{\hat g_r}  =
C  r^{-1}  \int_{M_r} R_{ g}  \, d V_{g} \\
= & \  o ( r^{-1}  ),
\end{split}
\ee
where in the last step, we used the assumption $R(g) $ is integrable on $(M, g)$. 

For the integral involving $ \nabla^2 u_r$, we have
\be
\begin{split}
\int_{\R^3} |  \nabla^2  u_r |^2_{\hat g_r} \, d V_{\hat g_r} 
\le C \left( \int_{\R^3} |  \p^2 u_r |^2_{g_0} \, d V_{g_0}  + 
\int_{\R^3} |  \hat \Gamma^{k}_{ij} |^2 | \nabla  u_r |^2_{g_0}  \, d V_{g_0}  
\right).
\end{split}
\ee
Here $\hat \Gamma^{k}_{ij}$ is the connection coefficients of $\hat g_r$.
The right side can be handled by H\"{o}lder's inequality.
Let $ \rho = \sqrt { 1 + | x|^2} $, then 
\be
\begin{split}
\int_{ \R^3 } |  \hat \Gamma^{k}_{ij} |^2   \, d V_{g_0}  
\le  & \ C \int_{ \R^3} |  \p \hat g_r |^2   \, d V_{g_0}  \\
= & \ C  \int_{\R^3} |  \p \hat g_r |^2 \rho^{ 2 ( \tau +1 )   - \frac{6}{q} }  \rho ^{ - 2 ( \tau +1 ) + \frac{6}{q} }  \, d V_{g_0}  \\
\le & \  C  || \p \hat g_r ||_{W^{0, p}_{- 1 -\tau} (\R^3) }^2  
\left( \int_{\R^3} \rho^{ \left[- 2 ( \tau + 1) + \frac{6}{q}  \right] \frac{q}{q-2} } \, d x \right)^{ 1 - \frac{2}{q} } \\
= & \ O (r^{ - 2 \tau} ),
\end{split}
\ee
where we used \eqref{eq-p-g-0-1}, \eqref{eq-p-g-B} and the assumption $ \tau > \frac12 $.

Similarly, using the fact $ \p^2 u_r = \p^2 w_r $, we have
\be \label{eq-Hessian-ur}
\begin{split}
\int_{\R^3} | \p^2 u_r |^2 \, d V_{g_0}   
= & \  \int_{\R^3} |  \p^2 w_r |^2 \rho^{ 2 ( \tau +1 )   - \frac{6}{q} }  \rho^{ - 2 ( \tau +1 ) +  \frac{6}{q} }  \, d V_{g_0}  \\
\le & \  || \p^2 w_r ||_{W^{0,p}_{- 1- \tau} ( \R^3 )}^2 
\left( \int_{\R^3}  \rho^{ [ - 2 ( \tau +1 ) + \frac{6}{q} ] \frac{q}{q-2}  }   \, d V_{g_0}  \right)^{ 1 - \frac{2}{q}} \\
= & \ O  ( r^{-2 \tau} ),
\end{split}
\ee
where we used  \eqref{eq-pw-1} in the last step.

It follows  from \eqref{eq-main-by} -- \eqref{eq-Hessian-ur} that
\be
\m (g)  = \frac{1}{ 8 \pi} \int_{\Sigma_r} (H_0 - H) \, d \sigma_r   + O (r^{1- 2 \tau}) + o(1) ,
\ee
which proves \eqref{eq-BY-limit-1}.
\end{proof}

\begin{remark}
In the statement of Theorem \ref{thm-mass-bymass}, we formulate the asymptotically flatness of $g$ 
via weighted Sobolev spaces $ W^{k,p }_{- \tau}$ with $ k \ge 4$ and $p > 3$. 
It is clear from the proof that one can also use a weighted H\"{o}lder space $C^{k, \alpha}_{- \tau} $ (see \cite{Bartnik}), 
with $ k \ge 3$, to specify the asymptotic of $g$. The fact we need a H\"{o}lder decay condition on the third order 
derivative of $g$ comes from estimate \eqref{eq-tXr-1} which is needed for the estimate of $\tilde \nu - \tilde \nu_0$ 
in \eqref{eq-gauge-F-nu}. 
\end{remark}

\vspace{.2cm}

\appendix

\section{The desingularization approach} \label{appen-disingularization}
We give another proof of  \eqref{eq-main-intro}  by combining the mass formula in \cite{BKKS} with the mollification method in \cite{M1}. 
Let $(E, g)$, $(\Omega, g)$, $(M, g)$ and $\Sigma$ be given as in the beginning of Section \ref{sec-mass-fill-in}. 
Applying Proposition 3.1 in \cite{M1} (and its proof therein) 
to $(M, g)$ near $\Sigma$, we have

\begin{proposition}\label{mol}
There exists a family of $C^2$ metrics $\{g_{\d}\}_{\d>0}$ on $M$  with the properties: 
\begin{enumerate}
    \item $g_\d=g$ outside $\S\times(-\d,\d)$; $ g_\d $ converges to $ g$ in $C^0$ on $M$;
    \item $||g_\d||_{C^{0,1}(M)}$ and thus $||\Gamma_{\d}||_{L^\infty(M)}$ are bounded uniformly with respect to $\delta$.
    Here $ \Gamma_\d $ denotes the Christoffel symbol of $g_\d$;
        \item $R_{\delta} (x, t) =O(1)$ in $\S\times \{ {\d^2}<|t|<\d\}$, and
    \item $R_{\delta} (x, t) =O(1)+ 2 (H_{_\Omega} -H)(x) \frac{1}{\d^2}  \phi \left( \frac{t}{\delta^2} \right)$ in $\S\times(- {\d^2}, \d^2)$.
\end{enumerate}
Here $\Sigma \times (-\delta, \delta)$ denotes a Gaussian tubular neighborhood of $\Sigma$ in $(M, g)$ 
which consists of points that are less than $\delta$ distance to $\Sigma$;  $t$ is the coordinate in $(-\d , \d)$ such that
$\p_t$ points to the infinity; $ R_\d (x, t) $ is the scalar curvature of $g_\d$ at a point $(x, t) \in \Sigma \times (- \d, \d )$; 
$O(1)$ is a bounded quantity that is independent on $\d$, and $\phi (t)$ is a standard mollifier on $ \R^1$.
\end{proposition}

Applying equation (6.28) from \cite{BKKS}  to $(M, g_\d)$,  using the fact $g_\d = g $ near infinity, we have

\begin{proposition}\label{IF}
Let $u_\d $ be the harmonic function on $(M, g_\d)$ which is asymptotic to $x_1$.   Then
\begin{equation}
\m (g)    \geq   16 \pi \int_M  \frac{|\nabla_\d^2 u_\d |^2}{|\Na_\d u_\d |}+R_{\delta} |\Na_\d u_\d |.
\end{equation}
Here $\nabla_\d$, $\nabla^2_\d$ denote the operators with respect to $g_\d$.
\end{proposition}

\vh

\begin{proof}[Proof of \eqref{eq-main-intro} using Propositions \ref{mol} and \ref{IF}]

\

\vh

On $(M, g)$,  let $ g_0 = g $ and let $ \{ g_\delta \}$ be given in Proposition \ref{mol}.
Let $ u_0 $ be the  harmonic function on $(M, g_0)$,  asymptotic to $x_1$ near infinity.
($u_0$ is precisely the function $u$ considered in Section \ref{sec-mass-fill-in}.)
Similar to the start of the proof of Proposition \ref{prop-mass-bdry}, 
we may  view this $u_0$ as one of the asymptotically flat  coordinate functions near infinity.
In what follows, we define $\bar x_1  = u_0 $.

For each $\delta > 0$, we recall how one finds the harmonic function $u_\delta $ that is asymptotic to $ \bar x_1 $.
First, we compute $ \Delta_{\delta} \bar x_1 $, the Laplacian of $\bar x_1$ with respect to $g_\d$, 
which satisfies 
$ \Delta_{\delta} \bar x_1 = 0 $  outside $ \Sigma \times (- \delta, \delta ) $.
This shows 
$ \Delta_{\delta} \bar x_1 \in W_{ - \eta - 2  }^{0 , p} (M)$,  $\forall \, \eta >  0 $.
Consider the operator 
\be
\Delta_{\delta} : W^{2, p}_{ - \eta} (M) \ \rightarrow \ W^{0,p}_{  - \eta - 2 } (M)  . 
\ee
By Proposition 2.2 in \cite{Bartnik}, $ \Delta_{\d}$ is an isomorphism provided $ 0 < \eta < 1 $.
Therefore, for a given $ \eta \in (0,1)$,  there exists a $ w_\delta \, \in W^{2,p}_{ - \eta  } (M ) $
such that $ \Delta_{\delta} w_\delta = \Delta_{\delta} \bar x_1$ and
\be
\begin{split}
|| w_\delta ||_{ W^{2,p}_{ - \eta  } (M ) } \le & \ C || \Delta_{\delta} w_\delta ||_{W^{0,p}_{-\eta -2} (M)} \\
= & \ C || \Delta_{\delta} \bar x_1 ||_{W^{0,p}_{-\eta -2} (M)} .
\end{split}
\ee
Here $C$ can be taken to be independent on $\delta$ by (i) and (ii) of Proposition \ref{mol}.

Since $ \bar x_1 = u_0 \in W^{2, p}_{loc} (M)$ and $ \{ g_\delta \}$ is  uniformly bounded  in $C^{0,1}(M)$, 
we have 
\bee
\int_{\overline U}   | \Delta_{\delta} \bar x_1 |^p < \Lambda  ,
\eee
where  $U$ is a fixed bounded open set containing $\Sigma$ and $\Lambda $ is a constant independent on $\delta$.
This combined with (i) of Proposition \ref{mol} shows, as $\delta \to 0$, 
\be
|| \Delta_{\delta} \bar x_1 ||_{W^{0,p}_{-\eta -2} (M)}  = o (1) , 
\ee
Therefore, 
\be \label{eq-w-d}
|| w_\delta ||_{ W^{2,p}_{ - \eta  } (M ) }  =  o (1) .
\ee
By Sobolev embedding,  $w_\d$ also satisfies  $ || w_\delta  ||_{ C^{1, \alpha}_{ - \eta} (M)  }  =  o (1) $.
Set $ u_\delta = u_0 - w_\delta $, $u_\delta$ is the harmonic function on $(M, g_\d)$ that is asymptotic to $\bar x_1$.

By Propositions \ref{mol} and \ref{IF}, we have

\begin{equation} \label{eq-A-m-d}
\begin{split}
\m (g) \geq &  \  \underbrace{  \int_{ M  }
\frac{|\Na^2  u_{\d}|^2}{|\Na  u_{\d}|} }_{\rom{1}}  
  + \underbrace{ \int_{ M  \setminus  \{ \Sigma \times (-\delta, \delta ) \} } R_\d  |\Na u_{\d}| }_{\rom{2} } 
 +  \underbrace{\int_{\S}\int_{-\delta}^{\delta} R_{\d}|\Na u_{\d}| }_{\rom{3}} .
\end{split}
\ee
Letting $\d \to 0 $,  by \eqref{eq-w-d} and Fatou's Lemma, 
\be
\liminf_{\d \to 0} \,  \rom{1}  
\ge \liminf_{\d \to 0}
  \int_{ M  \setminus  \{ \Sigma \times (-\delta, \delta ) \} }
\frac{|\Na^2  u_{\d}|^2}{|\Na  u_{\d}| }
\ge \int_{M } \frac{|\Na^2 u_0|^2} {|\Na u_0|} .
\ee
By \eqref{eq-w-d} and Proposition \ref{mol}, 
\be \label{eq-A-R-d}
\lim_{\d \to 0} \, \rom{2} =  \int_{M   } R |\Na u_{0}|,  \ \ 
\text{and} \ \ 
\lim_{\d \to 0} \, \rom{3} =  2 \int_\Sigma (H_{_\Omega} - H ) | \nabla u_0 |.
\ee
Thus \eqref{eq-main-intro} follows from \eqref{eq-A-m-d} -- \eqref{eq-A-R-d}.
\end{proof}

We remark that though this approach established \eqref{eq-main-intro}, it does not give the regularity of $u$ on $\Sigma$
which is a main component of Theorem \ref{main-1-intro}. 

\bigskip

\noindent {\bf Acknowledgements.}
SH thanks Hubert Bray for insightful discussions relating to the level set method. 
PM is grateful to Danial Stern for giving the geometric interpretation of the boundary term in \eqref{eq-angle-Stern}.
PM is also indebted to Siyuan Lu for
explaining estimates on isometric embedding pertinent to Theorem \ref{main-2-intro}.
TT  thanks  Man-Chun Lee and Richard Schoen for helpful discussions on the desingularisation approach.

The authors want to thank the anonymous referee for helpful suggestions, specially for
the insightful comment on the zero mass case in relation to Corollary \ref{cor-WY}.

\end{document}